\documentclass{amsart}

\usepackage[all,RR,hyperref]{bi-discrete}

\usepackage[backend=biber,maxbibnames=5,maxalphanames=5,style=alphabetic,bibencoding=utf8,giveninits,doi=true,url=false=true, maxcitenames=5]{biblatex}
\AtBeginBibliography{\small}

\makeatletter
\@namedef{subjclassname@2020}{\textup{2020} Mathematics Subject Classification}
\makeatother

\begin{document}

\author{Anna Kh.~Balci}

\address{Anna Kh.~Balci,  Charles University  in Prague, Department of Mathematical Analysis
Sokolovsk\'a 49/83, 186 75 Praha 8, Czech Republic and University of Bielefeld, Department of Mathematics,  Postfach 10 01 31, 33501 Bielefeld, Germany. }
\email{akhripun@math.uni-bielefeld.de}
\thanks{Anna Kh. Balci is funded by the Deutsche Forschungsgemeinschaft (DFG, German Research Foundation) - SFB 1283/2 2021 - 317210226, Charles University in Prague PRIMUS/24/SCI/020 and Research Centre program No. UNCE/24/SCI/005.
}
\author{Ho-Sik Lee}
\address{Ho-Sik Lee, University of Bielefeld, Department of Mathematics,  Postfach 10 01 31, 33501 Bielefeld, Germany.}
\email{ho-sik.lee@uni-bielefeld.de}
\thanks{Ho-Sik Lee is funded by Deutsche Forschungsgemeinschaft through GRK 2235/2 2021 - 282638148.}
\title{Zaremba problem with degenerate weights}

\keywords{nonlinear elliptic equations; degenerate weights; Zaremba problem; Meyers property; Muckenhoupt classes}

\subjclass[2020]{%
35J60, 
35B65, 
35J15.  
}

\begin{abstract}
We establish Zaremba problem for Laplacian and~$p$-Laplacian with degenerate weights when the Dirichlet condition is only imposed in a set of positive weighted capacity. We prove weighted Sobolev-Poincar\'{e} inequality with sharp scaling-invariant constants involving weighted capacity. Then we show higher integrability of the gradient of the solution (Meyers estimate) with minimal conditions on the part of the boundary where the Dirichlet condition is assumed. Our results are new both for the linear~$p=2$ and nonlinear case and include problems with the weight not only as a measure but also as a multiplier of the gradient of the solution.
\end{abstract}

\maketitle
\section{Introduction}

The purpose of this article is to obtain the small higher integrability (Meyers property) for the solution of   mixed boundary value problem -- Zaremba problem -- for the second order elliptic equation with degenerate weights.  The prototype equation is the weighted Laplacian of the form
\begin{align}\label{eq:lap}
\begin{split}
-\divergence\left(\mathbb{A}(x)\nabla u\right)&=l\quad\text{in }\Omega,\\
u&=0\quad\text{on }\mathcal{D},\\
\dfrac{\partial u}{\partial \nu}&=0\quad\text{on }\mathcal{N},
\end{split}
\end{align}
and weighted~$p$-Laplacian of the form
\begin{align}
  \label{eq:plap.multiplier}
  \begin{alignedat}{3}
    -\divergence(|\setM(x)\nabla u|^{p-2}\setM^2(x)\nabla u) &=l &\qquad &\text{in }\Omega,
    \\
    u &=0&\qquad&\text{on }\mathcal{D},\\
    \frac{\partial u}{\partial \nu}&=0&\qquad&\text{on }\mathcal{N}.
    \end{alignedat}
\end{align} 
Here, $1<p<\infty$,  $\Omega\subset\setR^n$ is a bounded Lipschitz   domain with $n\geq 2$. The set $\mathcal{D}\subset \partial\Omega$ is a closed set where the Dirichlet condition is assumed and~$\mathcal{N}=\partial \Omega\setminus \mathcal{D}$ is the part of the boundary where the Neumann condition is. We point out that the Dirichlet and Neumann conditions can irregularly be alternating on the boundary. The data $l$ in the right-hand side is a linear functional for suitable function spaces. The coefficient functions $\mathbb{A},\mathbb{M}:\setR^n\rightarrow\setR^{n\times n}_{\text{sym}}$ are symmetric and positive definite matrix-valued weights which satisfy $\setM^2(x)=\mathbb{A}(x)$ and
\begin{align}
\label{eq:M}
\abs{\setM(x)}\, \abs{\setM^{-1}(x)}\leq\Lambda
\end{align} 
for any $x\in\setR^n$ with some constant $\Lambda\geq 1$, where by~$\abs{\cdot}$ we denote the spectral norm of a matrix. Note that the condition \eqref{eq:M} is equivalent to say that $\bbM$ has a uniformly bounded condition number $\Lambda$. The typical example of such weight is $\setM(x)=|x|^{\pm\epsilon}\identity$ for small $\epsilon>0$. Note that when $p=2$, \eqref{eq:plap.multiplier} becomes linear problem \eqref{eq:lap}. To our knowledge no regularity results for Zaremba problem with possible degenerate weights are known both for the linear case~\eqref{eq:lap} and nonlinear case~\eqref{eq:plap.multiplier}. 

Let us start with several results concerning local and global results regarding Dirichlet problem for the linear weighted equations \eqref{eq:lap}. Fabes, Kenig and
Serapioni~\cite{FabKenSer82} studied the case of Dirichlet problem for linear weighted equation with degenerate coefficients $\mathbb{A}(x):\Omega\rightarrow \setR^{n\times n}_{\sym}$ such that
\begin{align}
  \label{eq:ass-fabes-kenig-serapioni}
  \Lambda^{-1} \mu(x) \abs{\xi}^2 \leq \skp{\bbA(x) \xi}{\xi} \leq
  \Lambda \mu(x) \abs{\xi}^2\quad(\xi\in\setR^n)
\end{align}
for some non-negative weight~$\mu$. We mention that \eqref{eq:ass-fabes-kenig-serapioni} is equivalent to say 
$\bbA(x)$ has a uniformly bounded condition number~$\Lambda^2$. The authors proved boundedness and H\"older continuity of the solution if~$\mu$ belongs to the Muckenhoupt class~$\mathcal{A}_2$ and $l$ is sufficiently regular. In \cite{CMP1}, the authors showed $|F|^q\mu\in L^{1}\Rightarrow |\nabla u|^q\mu\in L^{1}$ for all $q\in(1,\infty)$, where the conditions $\mu\in\mathcal{A}_2$ and a smallness assumption on a weighted $\setBMO$ norm of $\mathbb{A}$ were imposed, but still included the case $\mu(x)=|x|^{\pm\epsilon}\identity$ for small $\epsilon>0$. See also \cite{P1,CMP2} for the related results.

In the recent paper \cite{BalDieGioPas22}, the authors proved a slightly different type of gradient estimates like $(|F|\omega)^q\in L^{1}_{\loc}\Rightarrow (|\nabla u|\omega)^q\in L^{1}_{\loc}$ for each $q\in(1,\infty)$ with the assumptions \eqref{eq:M} and the smallness condition  
\begin{align*}
[\log\mathbb{A}]_{\setBMO_{\loc}(\Omega)}=\sup_{B\Subset\Omega}\dashint_{B}\lvert\log\mathbb{A}(x)-\mean{\log\mathbb{A}}_{B}|\,dx\leq\frac{c}{q}
\end{align*}
for some $c=c(n,p,\Lambda)$, where $B$ denotes the ball in $\setR^n$ and  $\mean{f}_B=\frac{1}{|B|}\int_{B}f\,dx$ for an integrable function $f:\setR^n\rightarrow\setR^{n\times n}_{\sym}$. Here, since $\mathbb{A}$ is positive definite almost everywhere, the logarithm $\log\mathbb{A}:\setR^n\rightarrow\setR^{n\times n}_{\sym}$ is well-defined. This $\log$--$\setBMO$ condition for $\mathbb{A}$ also includes the case $\mathbb{A}(x)=|x|^{\pm\epsilon}\identity$ for small $\epsilon>0$. Moreover, the condition is optimal for integrability exponent $q$. In \cite{BBDL23}, Balci, Byun, Diening and Lee show the same result up to the boundary for the Dirichlet boundary value problem. In particular, that higher integrability of the gradients can be obtained by imposing a local small oscillation condition on the weight and a local small Lipschitz condition on the boundary of the domain.

In this paper we concentrate on Zaremba  problems \eqref{eq:lap} and~\eqref{eq:plap.multiplier}. Mixed boundary value problems admit less regularity properties for the general case of Zaremba problem.  Without extra assumptions on the angles between the subsets of boundary where the Dirichlet condition and Neumann condition are imposed, the only available regularity results are small higher integrability of the gradient (Meyers property), which means that~$\nabla u \in L^{p+\epsilon}(\Omega)$ for small $\epsilon>0$ in the nonlinear case, which is above the natural a-priori integrability.   In  \cite{AlkCheMaz22} Alkhutov, Chechkin and  Maz'ya obtained Meyers estimate for solutions to the linear {Z}aremba problem with uniformly elliptic coefficients , i.e.
$$\lambda_1 \abs{\xi}^2 \leq \skp{\bbA(x) \xi}{\xi} \leq \Lambda_1
\abs{\xi}^2$$ for all~$x \in \Omega$ with some $0<\lambda_1\leq\Lambda_1$, and  in
\cite{AlkChe22} Alkhutov and Chechkina obtained Meyers property for  unweigthed ~$p$-Laplace equation. See also   \cite{ADKC23}, \cite{AC23}, and \cite{AlkChe24} for the other contributions for different models including~$p(\cdot)$-Laplacian.    Due to the assumed uniform ellipticity in previous papers, the results above exclude the possibility of degenerate weights like
$\abs{x}^{\pm \epsilon} \identity$. In the present paper, we study the linear and nonlinear problems with degenerate weights that belong to the natural Muckenhoupt classes~$\mathcal{A}_p$ under the setting of the Zaremba problem, and obtain Meyers property for the solution. 

The method of the proof goes back to Maz'ya~\cite{Maz11} and \textcite{AlkCheMaz22}. It is based on the several general results including Sobolev-\Poincare inequality with conditions formulated in terms of capacity. In our paper, assuming \eqref{eq:ass-fabes-kenig-serapioni}, we prove such results for our weighted problem with $\mu(x)\in\mathcal{A}_p$, and then we show existence and uniqueness of the weak solution to our problems with the assumption \eqref{eq:cap0}, where the weighted capacity is defined in \eqref{eq:capa}. Then assuming scaling-invariant condition for capacity such as \eqref{eq:cen}, which in principle implies that $\mathcal{D}$ satisfies $\mu$-weighted essential subset of $Q_r$ and $\mathcal{D}$ is not negligible in the sense of $\mu$-weighted capacity, the following
\begin{align}\label{eq:result}
\abs{G}^{p+\delta}\mu\in L^{1}(\Omega)\quad\implies\quad \abs{\nabla u}^{p+\delta}\mu\in L^{1}(\Omega)
\end{align}
is shown for any $\delta\in[0,\delta_0]$ with small $\delta_0$, when the data in the right-hand side of the equation is expressed as $-\divergence(|G|^{p-2}\mathbb{A}(x)G)$. We also prove the analogous result to \eqref{eq:result} for the problem \eqref{eq:plap.multiplier} as follows:
\begin{align}\label{eq:result'}
\abs{\setM G}^{p+\delta}\in L^{1}(\Omega)\quad\implies\quad \abs{\setM\nabla u}^{p+\delta}\in L^{1}(\Omega)
\end{align}
when $l=-\divergence(|\setM G|^{p-2}\setM(x)^2G)$ for any $\delta\in[0,\delta_0]$ with small $\delta_0$.

The outline of this article is as follows. In Section~\ref{sec:2} we provide basic notations, definitions, and lemmas for our results. Section~\ref{sec:3} is organized as follows:
\begin{itemize}
	\item We prove weighted Sobolev-Poincar\'{e} inequality with sharp scaling-invariant capacity conditions.
	\item Then we introduce weighted problems and show existence results.
	\item Finally, we prove higher integrability estimates for the gradient of the solution.
\end{itemize}
In Section~\ref{sec:4}, we give analogous results for the weighted equations where the weights play roles as multipliers of the gradient of the solutions.

\section{preliminaries}\label{sec:2}

We use $c$ as generic constants which can vary from each line of the paper, and we clarify its dependence by using parenthesis, e.g., $c=c(n,p)$ means that $c$ depends on $n$ and $p$. For (real-valued) functions or numbers $f$ and $g$, we say $f\lesssim g$ when there exists a implicit constant $c$ such that $f\leq cg$ holds. We denote $f\eqsim g$ when both $f\lesssim g$ and $g\lesssim f$ holds, i.e., there exists a implicit constant $c\geq 1$ such that $\frac{1}{c}g\leq f\leq cg$ holds.
 
For $r>0$ and $x_0=(x_{0,1},\dots,x_{0,n})\in\setR^n$, we define the cube
\begin{align*}
Q_r(x_0)=\{x=(x_1,\dots,x_n)\in\setR^n:|x_i-x_{0,i}|<r\,\,\text{for each }i\}.
\end{align*}
Also, let us denote by $Q^+_r(x_0):=Q_r(x_0)\cap\setR^n_+$. We occasionally omit the center of the cube if it is not important in the context. We write the integral average of a measurable function $f:Q_r(x_0)\rightarrow \setR$ as
\begin{align*}
\mean{f}_{Q_r(x_0)}=\dashint_{Q_r(x_0)}f\,dx=\dfrac{1}{|Q_{r}(x_0)|}\int_{Q_{r}(x_0)}f\,dx.
\end{align*}

A function $\mu\in L^1_{\loc}(\setR^n)$ is called a weight if $\mu>0$ almost everywhere in $\setR^n$. For $1<p<\infty$ and an open set $U\subset\setR^n$, a weight $\mu\in L^{1}_{\loc}(\setR^n)$ is called an $\mathcal{A}_p(U)$-Muckenhoupt weight, i.e., $\mu\in\mathcal{A}_p(U)$, if
\begin{align*}
[\mu]_{\mathcal{A}_p(U)}:= \sup_{Q_r(x_0)\subset U}\left(\dashint_{Q_r(x_0)}\mu\,dx\right)\left(\dashint_{Q_r(x_0)}\mu^{-\frac{1}{p-1}}\,dx\right)^{p-1}<\infty.
\end{align*}
We remark that the $\mathcal{A}_p$-Muckenhoupt weight is originally defined with the ball $B_r(x_0)$ instead of $Q_r(x_0)$, but since $B_r(x_0)\subset Q_r(x_0)\subset B_{\sqrt{2}r}(x_0)$, $|B_r(x_0)|\eqsim |Q_r(x_0)| \eqsim |B_{\sqrt{2}r}(x_0)|$ and $\mu>0$ a.e., our definition written with $Q_r(x_0)$ is equivalent to the original one. With $Q_{r}\subset U$, we write 
\begin{align*}
\mu(Q_r)=\int_{Q_r}\mu(x)\,dx.
\end{align*}
Now we list some basic properties when $\mu$ is an~$\mathcal{A}_p(U)$-Muckenhoupt weight.
\begin{itemize}
\item The maximal operator~$M$ is bounded on $L^p(U;\mu dx)$ (see \cite{Muc72}), where $L^p(U;\mu dx)$ is the completion of functions which are infinitely differentiable in $\overline{U}$ endowed with
\begin{align*}
\norm{v}_{L^{p}(U;\mu dx)}=\left(\int_{U}|v|^p\mu\,dx\right)^{\frac{1}{p}}.
\end{align*}
\item By \cite[Proposition 7.2.8]{Gra14class}, for $U\subset Q$ with any cube $Q\subset\setR^n$ we have
\begin{align}\label{eq:UB}
\dfrac{\mu(U)}{\mu(Q)}\leq c\left(\dfrac{|U|}{|Q|}\right)^{\delta}
\end{align}
for some $\delta=\delta(n,p,[\mu]_{\mathcal{A}_p})\in(0,1)$ and $c=c(n,p,[\mu]_{\mathcal{A}_p})\geq 1$.
\item Moreover, by \cite[Lemma 3.5]{MenPhu12}, there holds
\begin{align}\label{eq:BU}
\dfrac{\mu(Q)}{\mu(U)}\leq[\mu]_{\mathcal{A}_p}\left(\dfrac{|Q|}{|U|}\right)^p
\end{align}
for $U\subset Q$ with $|U|>0$.
\item By \cite[Lemma 3.6]{MenPhu12}, there exists $\gamma_1=\gamma_1(n,p,[\mu]_{\mathcal{A}_p})\in(1,p)$ such that  
\begin{align}\label{eq:mu.gamma}
[\mu]_{\mathcal{A}_{\gamma}}\leq c(n,p,[\mu]_{\mathcal{A}_p})
\end{align}
for any $\gamma\in[\gamma_1,p]$. 
\end{itemize}
For the further studies of Muckenhoupt weight, we refer to \cite[Chapter 7]{Gra14class}.

Now, we introduce Lipschitz domain.
\begin{definition}\label{def:lip}
	Let two positive parameters $L$ and $R$ be given. The domain $\Omega$ is called $(L,R)$--Lipschitz if for any $x_0\in\partial\Omega$, we can find a coordinate system $\{x_1,\dots,x_n\}$ and Lipschitz map $\psi:\setR^{n-1}\rightarrow\setR$ such that $x_0$ is the origin in this new coordinate system, and we have
	\begin{align}\label{eq:lip1}
	\Omega\cap Q_R(x_0)=\{x=(x_1,\dots,x_n)=(x',x_n)\in Q_R(x_0):x_n>\psi(x')\}
	\end{align}
	together with
	\begin{align}\label{eq:psi}
	\|\nabla\psi\|_{\infty}\leq L.
	\end{align}
\end{definition}
One of the good property of an Lipschitz domain is that Lipschitz domain is an extension domain for not only Sobolev spaces but also weighted Sobolev spaces (See \cite{Chu92}). This property is used for Corollary \ref{cor:SP} and Corollary \ref{cor:SP'}, where derive the existence of the weak solution to our problem in consideration.

Let $1<p<\infty$, $\Omega\subset\setR^n$ be a bounded $(L,R)$--Lipschitz domain with $n\geq 2$, $\mathcal{D}\subset\partial\Omega$ be a closed set and $\mu\in \mathcal{A}_p(\setR^n)$. 
We denote by $W^{1,p}(\Omega,\mathcal{D};\mu dx)$ the completion of functions which are infinitely differentiable in $\overline{\Omega}$ and equal to zero in a neighborhood of $\mathcal{D}$ with the norm
\begin{align*}
\norm{v}_{W^{1,p}(\Omega,\mathcal{D};\mu dx)}=\left(\int_{\Omega}|v|^p\mu\,dx+\int_{\Omega}|\nabla v|^p\mu\,dx\right)^{\frac{1}{p}}.
\end{align*}
Note that it is possible to define $W^{1,p}(\Omega,\mathcal{D};\mu dx)$ in this way since smooth functions are dense in weighted Sobolev spaces with Muckenhoupt weights.

For an open set $U\subset\setR^n$, $K_0$ being closed subset of $U$, $\mu\in\mathcal{A}_q$ for $q\geq 1$, let us define the weighted capacity which will be frequently used throughout the paper as follows:
\begin{align}\label{eq:capa}
\textrm{Cap}_{q,\mu}(K_0,U):=\inf\left\{\int_{U}|\nabla\phi|^q\mu\,dx:\phi\in C^{\infty}_0(U),\,\,\phi\geq 1\,\,\text{on}\,\, K_0\right\}.
\end{align}

Let us mention that there is a fundamental relation between the capacity of a set  and validity of the Sobolev-Poincar\'{e} inequality. Here we record some results which go back to classic works by Maz'ya, see \cite[Chapters 2,13,14]{Maz11} and \cite{HeiKilMar06} and references therein.

We now give weighted Sobolev-Poincar\'{e} inequality with the appropriate dependency to the universal constants for the later use.
\begin{lemma}\label{lem:SP0}
For $1<p<\infty$ and $\mu(\cdot)\in\mathcal{A}_p$, there exist $c=c(n,p,[\mu]_{\mathcal{A}_p})\geq 1$ and $q_0=q_0(n,p,[\mu]_{\mathcal{A}_p})<p$ such that 
\begin{align}\label{eq:SP0}
\left(\dfrac{1}{\mu(Q_r)}\int_{Q_r}\left|\dfrac{u}{r}\right|^p\mu\,dx\right)^{\frac{1}{p}}\leq c \left(\dfrac{1}{\mu(Q_r)}\int_{Q_r}|\nabla u|^q\mu\,dx\right)^{\frac{1}{q}}
\end{align}
for any $q\in[q_0,p]$, $Q_r\subset\setR^n$ and $u\in W^{1,q}_{0,\mu}(Q_r)$. Moreover, we have
\begin{align}\label{eq:SP1}
\left(\dfrac{1}{\mu(Q_r)}\int_{Q_r}\left|\dfrac{u-A}{r}\right|^p\mu\,dx\right)^{\frac{1}{p}}\leq c \left(\dfrac{1}{\mu(Q_r)}\int_{Q_r}|\nabla u|^q\mu\,dx\right)^{\frac{1}{q}},
\end{align}
where $A=\mean{u}_{Q_r}$ or $A=\frac{1}{\mu(Q_r)}\int_{Q_r}u(x)\mu\,dx$.
\end{lemma}
\begin{proof}
The proof is based on \cite{FabKenSer82} together with \cite{MucWhe74}. One of the differences compared to the mentioned literatures is that we use the cube $Q_r$ instead of the ball $B_r$, but in principle the proof will be same. First, similar to \cite[Theorem 1]{MucWhe74} and \cite[Page 84]{FabKenSer82}, we see that the following is true: For $1<p<\infty$ and $\mu\in\mathcal{A}_p$, there exists a constant $c=c(n,p,[\mu]_{\mathcal{A}_p})>0$ such that for all measurable $f:\setR^n\rightarrow\setR$,
\begin{align}\label{eq:bddT}
\int_{\setR^n}\left|\int_{\setR^n}\dfrac{f(y)}{|x-y|^{n-1}}\,dy\right|^p\mu(x)\,dx\leq c\int_{\setR^n}|Tf(x)|^p\mu(x)\,dx
\end{align} 
holds, where
\begin{align*}
Tf(x)=\sup_{0<s<\infty}\dfrac{1}{s^{n-1}}\int_{Q_s(x)}|f(y)|\,dy.
\end{align*}
Indeed, using \eqref{eq:UB} and following the argument of \cite[Page 263--265]{MucWhe74}, we obtain \eqref{eq:bddT}.

On the other hand, similar to \cite[Lemma 1.1]{FabKenSer82}, the following holds true: There exists $c=c(n,p,[\mu]_{\mathcal{A}_p})\geq 1$ and $q_1=q_1(n,p,[\mu]_{\mathcal{A}_p})\in(1,p)$ such that
\begin{align}\label{eq:bddT2}
\left(\dfrac{1}{\mu(Q_r)}\int_{Q_r}(Tf)^p\mu(x)\,dx\right)^{\frac{1}{p}}\leq cr\left(\dfrac{1}{\mu(Q_r)}\int_{Q_r}|f|^q\mu(x)\,dx\right)^{\frac{1}{q}}
\end{align}
for $f:\setR^n\rightarrow \setR$ with $f\in L^q_{\mu}(Q_r)$ supported in the ball $Q_r$ and any $q\in[q_1,p]$. Indeed, with $\gamma_1$ in \eqref{eq:mu.gamma} we set
\begin{align*}
q_0=\max\left\{\frac{n-1}{n}p,\frac{p+\gamma_1}{2}\right\}\in(\gamma_1,p).
\end{align*}
Then following the same argument of the proof of \cite[Lemma 1.1]{FabKenSer82}, together with the fact that
\begin{align}\label{eq:doubling}
\dfrac{\mu(Q_r)}{\mu(E)}\leq [\mu]_{\mathcal{A}_{q}}\left(\dfrac{|Q_r|}{|E|}\right)^q\leq  [\mu]_{\mathcal{A}_{\gamma_1}}\left(\dfrac{|Q_r|}{|E|}\right)^q\leq c(n,p,[\mu]_{\mathcal{A}_p})\left(\dfrac{|Q_r|}{|E|}\right)^q
\end{align}
for any subset $E\subset Q_r$, we have \eqref{eq:bddT2}.

Now note that for any $\gamma\in[\gamma_1,p]$, $\mu\in \mathcal{A}_{\gamma}$ implies $\mu^{-\frac{1}{\gamma-1}}\in\mathcal{A}_{\frac{\gamma}{\gamma-1}}$ and so by \cite[Remark 7.2.3 and Proposition 7.2.8]{Gra14class}, we have
\begin{align*}
\dfrac{(\mu^{-\frac{1}{\gamma-1}})(E)}{(\mu^{-\frac{1}{\gamma-1}})(Q_r)}\leq c(n,p,[\mu]_{\mathcal{A}_p})\left(\dfrac{|E|}{|Q_r|}\right)^{\delta}
\end{align*}
with $E\subset Q_r$. Then following the argument of the proof of \cite[Theorem 1.2]{FabKenSer82} along with the above inequality, \eqref{eq:bddT} and \eqref{eq:bddT2}, we obtain \eqref{eq:SP0}. Now \eqref{eq:SP1} follows from the same argument of the proof of \cite[Theorem 1.5]{FabKenSer82}. 
\end{proof}

We also give weighted Sobolev-Poincar\'{e} inequality with the weight as a multiplier for appropriate dependency.

\begin{lemma}\label{lem:SP0'}
	For $1<p<\infty$ and $\mu(\cdot)^p\in\mathcal{A}_p$, there exist $c=c(n,p,[\mu^p]_{\mathcal{A}_p})\geq 1$ and $q_1=q_1(n,p,[\mu^p]_{\mathcal{A}_p})<p$ such that 
	\begin{align}\label{eq:SP0'}
	\left(\dashint_{Q_r}\left|\dfrac{u}{r}\right|^p\mu^p\,dx\right)^{\frac{1}{p}}\leq c \left(\dashint_{Q_r}|\nabla u|^q\mu^q\,dx\right)^{\frac{1}{q}}
	\end{align}
	for any $q\in[q_1,p]$, $Q_r\subset\setR^n$ and $u\in W^{1,q}_{0,\mu}(Q_r)$. Moreover, we have
	\begin{align}\label{eq:SP1'}
	\left(\dashint_{Q_r}\left|\dfrac{u-\mean{u}_{Q_r}}{r}\right|^p\mu^p\,dx\right)^{\frac{1}{p}}\leq c \left(\dashint_{Q_r}|\nabla u|^q\mu^q\,dx\right)^{\frac{1}{q}}.
	\end{align}
\end{lemma}
\begin{proof}
By \cite[Theorem 7.2.2]{Gra14class} there exists $\theta=\theta(n,p,[\mu^p]_{\mathcal{A}_p})\in(0,1)$ such that 
\begin{align*}
\left(\dashint_{Q_{r}}\mu^{-(\theta p)'}\,dx\right)^{\frac{1}{(\theta p)'}}\leq c\left(\dashint_{Q_{r}}\mu^{- p'}\,dx\right)^{\frac{1}{p'}}
\end{align*}
for some $c=c(n,p,[\mu^p]_{\mathcal{A}_p})$.
Using this, the inequality \eqref{eq:SP0'} can be proved by \cite[Proposition 2.2]{BBDL23} with the covering argument, and \eqref{eq:SP1'} can be shown by \cite[Proposition 3]{BalDieGioPas22}.
\end{proof}

\section{Zaremba problem with degenerate weights as measures}\label{sec:3}

This section is devoted to establishing the weighted equation where the weights are considered as measures. First of all, we prove optimal Sobolev-Poincar\'{e} inequality where the relevant constant is proportional to a scaling-invariant quantity involving the weighted capacity as in \eqref{eq:cap} and \eqref{eq:cap'}. Then we introduce precise settings of our problem, and show the existence and uniqueness result by assuming that the portion of the boundary where the Dirichlet condition is imposed has positive weighted capacity. 

Moreover, assuming the additional condition using the weighted capacity, a higher integrability result is proved. Finally, we give an nontrivial example of the set which satisfies our weighted capacity conditions.

\subsection{Weighted Sobolev-Poincar\'{e} inequality with optimal constants}\label{sec:3.1}
In this subsection, we prove weighted Sobolev-Poincar\'{e} inequality for our problem, together with finding the scaling-invariant quantities comparable to the optimal constants of the inequality. We also show that the aforementioned comparability is unavoidable in general. To do this, we first define the following notion for subsets in $\overline{Q}_r\subset\setR^n$.
\begin{definition}
Let $K$ be a closed subset of $\overline{Q}_r$,	$\gamma\in(0,\frac{1}{2})$ and $q>1$ be chosen numbers.
\begin{itemize}
	\item We say $K$ is a $\mu$-weighted ($\gamma$,$q$)-negligible subset of $\overline{Q}_r$, if
	\begin{align}\label{eq:neg}
	\dfrac{\textrm{Cap}_{q,\mu}(K,Q_{2r})}{\mu(Q_{2r})}\leq \gamma r^{-q}.
	\end{align}
	The collection of all $\mu$-weighted ($\gamma$,$q$)-negligible subset of $\overline{Q}_r$ is called $\mathcal{N}_{\mu;\gamma,q}(Q_r)$. 
	\item If \eqref{eq:neg} fails for the set $K$, then we say $K$ is a $\mu$-weighted ($\gamma$,$q$)-essential subset of $\overline{Q}_r$.
\end{itemize}
\end{definition}
Now we give the following proposition inspired by \cite[Section 14.1.2]{Maz11}.

\begin{proposition}\label{prop:SP}
Let $K$ be a closed subset of $\overline{Q}_r$ and $\mu\in\mathcal{A}_p(\setR^n)$ with $1<p<\infty$.
\begin{enumerate}
\item There exists $p_0=p_0([\mu]_{\mathcal{A}_p})\in(1,p)$ such that for each $q\in [p_0,p]$,
\begin{align}\label{eq:wSP}
\left(\dfrac{1}{\mu(Q_r)}\int_{Q_r}\left|\dfrac{u}{r}\right|^p\mu\,dx\right)^{\frac{1}{p}}\leq C\left(\dfrac{1}{\mu(Q_r)}\int_{Q_r}|\nabla u|^q\mu\,dx\right)^{\frac{1}{q}}
\end{align}
holds for all $u\in C^{\infty}(\bar{Q}_r)$ with $\text{dist}(\text{supp}\,u,K)>0$. Moreover, $C$ satisfies the following estimate
\begin{align}\label{eq:cap}
C\leq \frac{c_2}{r} \left(\dfrac{\mu(Q_{2r})}{\text{Cap}_{q,\mu}(K,Q_{2r})}\right)^{\frac{1}{q}}
\end{align}
for some $c_2=c_2(n,p,[\mu]_{\mathcal{A}_p})$.
\item For any $u\in C^{\infty}(\bar{Q}_r)$ with $\text{dist}(\text{supp}\,u,K)>0$, suppose that
\begin{align}\label{eq:wSP'}
\left(\dfrac{1}{\mu(Q_{\frac{r}{2}})}\int_{Q_{\frac{r}{2}}}\left|\frac{u}{r}\right|^p\mu\,dx\right)^{\frac{1}{p}}\leq C\left(\dfrac{1}{\mu(Q_{r})}\int_{Q_r}|\nabla u|^q\mu\,dx\right)^{\frac{1}{q}}
\end{align}
holds for some $q\in[p_0,p]$. Then there exists a small $\gamma=\gamma(n,p,[\mu]_{\mathcal{A}_p})\leq \frac{1}{2}$ such that if $K\in \mathcal{N}_{\mu;\gamma,q}(Q_r)$, we have
\begin{align}\label{eq:cap'}
C\geq \frac{c_3}{r} \left(\dfrac{\mu(Q_{2r})}{\text{Cap}_{q,\mu}(K,Q_{2r})}\right)^{\frac{1}{q}}
\end{align}
with $c_3=c_3(n,p,[\mu]_{\mathcal{A}_p})$.
\end{enumerate}
\end{proposition}
\begin{proof}
We first prove (a). Let us prove (a) in the case of $r=1$ and then use the scaling to prove in the case of $r\neq 1$. Note that from \cite[Lemma 3.6]{MenPhu12}, we can find
\begin{align}\label{eq:p0}
p_0=p_0(n,p,[\mu]_{\mathcal{A}_p})\in(1,p)
\end{align}
such that $\mu\in\mathcal{A}_p$ implies $\mu\in\mathcal{A}_{\gamma}$ for any $\gamma\in[p_0,p]$. Moreover, we have $[\mu]_{\mathcal{A}_{\gamma}}\leq c(n,p,[\mu]_{\mathcal{A}_p})$.

\textbf{Step 1.} We first show that if $K\in \overline{Q}_1$ is a compact set, then
\begin{align}\label{eq:cap1}
\begin{split}
&\text{Cap}_{q,\mu}(K,Q_2)\\
&\quad\eqsim\inf\left\{\sum^{1}_{m=0}\norm{\nabla_m(1-u)}^q_{L^q(Q_1;\mu dx)}:u\in C^{\infty}(\overline{Q}_1),\text{dist}(\text{supp}\,u,K)>0\right\}
\end{split}
\end{align}
for any $q\in[p_0,p]$ with the implicit constant $c=c(n,p,[\mu]_{\mathcal{A}_p})$, where we denote $\nabla_0 u=u$.
	To show this, we recall that by \cite[Section 1.1.17]{Maz11}, there exists a linear mapping
	\begin{align*}
		A:C^{k-1,1}(\bar{Q}_1)\rightarrow C^{k-1,1}(\bar{Q}_{2})\quad\text{for}\,\, k=1,2,\dots
	\end{align*}
	such that
	\begin{itemize}
		\item[(i)] $Av=v$ on $\bar{Q}_1$,
		\item[(ii)] if $\text{dist}(\text{supp}\,v,K)>0$, then $\text{dist}(\text{supp}(Av),K)>0$ and
		\item[(iii)] $\norm{\nabla_m Av}_{L^{\gamma}(Q_{2};\mu dx)}\leq c\norm{\nabla_mv}_{L^{\gamma}(Q_1;\mu dx)}$
	\end{itemize}
	for $m=0,1$ and $1\leq \gamma\leq \infty$.

	To show \textbf{Step 1}, let $v=A(1-u)$ and $\eta\in C^{\infty}_0(Q_{2})$ with $\eta\equiv 1$ in a neighborhood of $Q_1$. Then by the following triangle inequality
	\begin{align}\label{eq:tri}
		(a+b)^q\leq 2^q(a^q+b^q)\leq c(p)(a^q+b^q),
	\end{align}
	there holds
	\begin{align}\label{eq:cap2}
		\begin{split}
			\text{Cap}_{q,\mu}(K,Q_2)&\leq \int_{Q_2}|\nabla(\eta v)|^q\mu\,dx\\
			&\leq c\sum^1_{m=0}\norm{\nabla_mA(1-u)}^q_{L^q(Q_2;\mu dx)}\leq c\sum^1_{m=0}\norm{\nabla_m(1-u)}^q_{L^q(Q_1;\mu dx)}
		\end{split}
	\end{align}
	with $c=c(p)$, where for the last inequality we have used (i) and (iii). On the other hand, let
	\begin{align*}
		\mathcal{M}(K,Q_2):=\{f\in C^{\infty}_0(Q_2):f\equiv 1\,\,\text{in a neighborhood of }K\}.
	\end{align*} 
	Then by the weighted Poincar\'{e} inequality \eqref{eq:SP0} of Lemma \ref{lem:SP0},
	\begin{align*}
		\sum^1_{m=0}\norm{\nabla_m f}^q_{L^q(Q_1;\mu dx)}\leq c\norm{\nabla f}^q_{L^q(Q_2;\mu dx)}
	\end{align*}
	with $c=c(n,p,[\mu]_{\mathcal{A}_p})$. Minimizing the right-hand side of the above estimate over the set $\mathcal{M}(K,Q_2)$, we have
	\begin{align*}
		\sum^1_{m=0}\norm{\nabla_mf}^q_{L^q(Q_1;\mu dx)}\leq c\,\text{Cap}_{q,\mu}(K,Q_2).
	\end{align*}
	Then by the above estimate and \eqref{eq:cap2}, \textbf{Step 1} is proved.

	\textbf{Step 2.} Now we will show (a) when $r=1$.
	Let $u\in C^{\infty}(\bar{Q}_1)$, $\text{dist}(\text{supp}\,u,K)>0$ and the number $N$ be such that
	\begin{align*}
		N^q=\dfrac{1}{\int_{Q_1}\mu\,dx}\int_{Q_1}|u|^q\mu\,dx.
	\end{align*}
	Then by \textbf{Step 1} and the triangle inequality \eqref{eq:tri} we have
	\begin{align*}
		\text{Cap}_{q,\mu}(K,Q_2)&\leq \sum^1_{m=0}\norm{\nabla_m(1-N^{-1}u)}^q_{L^q(Q_1;\mu dx)}\\
		&=cN^{-q}\int_{Q_1}|\nabla u|^q\mu\,dx+c\int_{Q_1}|(1-N^{-1}u)|^q\mu\,dx
	\end{align*}
	with some $c=c(p)$, which is equivalent to
	\begin{align}\label{eq:cap3}
		N^q\text{Cap}_{q,\mu}(K,Q_2)\leq c\int_{Q_1}|\nabla u|^q\mu\,dx+c\int_{Q_1}|(N-u)|^q\mu\,dx.
	\end{align}
	We may assume that $\mean{u}_{Q_1}\geq 0$ without loss of generality. Then we obtain
	\begin{align*}
		|N-\mean{u}_{Q_1}|&=\left|\left(\dfrac{1}{\int_{Q_1}\mu\,dx}\int_{Q_1}|u|^q\mu\,dx\right)^{\frac{1}{q}}-\dashint_{Q_1}u\,dx\right|\\
		&\leq c\left(\dfrac{1}{\int_{Q_1}\mu\,dx}\int_{Q_1}|u-\mean{u}_{Q_1}|^q\mu\,dx\right)^{\frac{1}{q}}.
	\end{align*}
	Thus it follows that
	\begin{align}\label{eq:cap4}
		\begin{split}
			\int_{Q_1}|N-u|^q\mu\,dx&\lesssim\int_{Q_1}|N-\mean{u}_{Q_1}|^q\mu\,dx+\int_{Q_1}|u-\mean{u}_{Q_1}|^q\mu\,dx\\
			&\lesssim\int_{Q_1}|u-\mean{u}_{Q_1}|^q\mu\,dx.
		\end{split}
	\end{align}

	By \eqref{eq:cap3}, \eqref{eq:cap4} and the following weighted Poincar\'{e} inequality \eqref{eq:SP1} of Lemma \ref{lem:SP0},
	\begin{align*}
		\int_{Q_1}|u-\mean{u}_{Q_1}|^q\mu\,dx\leq c\int_{Q_1}|\nabla u|^q\mu\,dx,
	\end{align*}
	we find
	\begin{align}\label{eq:cap4.11}
		N^q\text{Cap}_{q,\mu}(K,Q_2)\leq c\int_{Q_1}|\nabla u|^q\mu\,dx
	\end{align}
	with $c=c(n,p,[\mu]_{\mathcal{A}_p})$.
	Then we have
	\begin{align}\label{eq:cap5}
	\begin{split}
\left(\int_{Q_1}|u|^p\mu\,dx\right)^{\frac{q}{p}}&\leq c\mu(Q_1)^{\frac{q}{p}-1}\int_{Q_1}|\nabla u|^q\mu\,dx+c\mu(Q_1)^{\frac{q}{p}-1}\int_{Q_1}|u|^q\mu\,dx\\
&\leq c\mu(Q_1)^{\frac{q}{p}-1}\left(1+\dfrac{\int_{Q_1}\mu\,dx}{\text{Cap}_{q ,\mu}(K,Q_2)}\right)\int_{Q_1}|\nabla u|^q\mu\,dx\\
&\leq c\mu(Q_1)^{\frac{q}{p}-1}\left(\dfrac{\int_{Q_2}\mu\,dx}{\text{Cap}_{q ,\mu}(K,Q_2)}\right)\int_{Q_1}|\nabla u|^q\mu\,dx\\
&\leq c\left(\dfrac{\mu(Q_1)^{\frac{q}{p}}}{\text{Cap}_{q ,\mu}(K,Q_2)}\right)\int_{Q_1}|\nabla u|^q\mu\,dx,	
\end{split}
\end{align}
where 
\begin{itemize}
	\item for the first inequality we have used  \eqref{eq:SP1} of Lemma \ref{lem:SP0},
	\item for the second one we employed \eqref{eq:cap4.11},
	\item for the third one we have used the fact that for any $\phi_0\in C^{\infty}_c(Q_2)$ with $\phi_0\equiv1$ in $Q_1$, $|\nabla\phi_0|\leq 2$ in $Q_2$, there holds
	\begin{align*}
	\text{Cap}_{q,\mu}(K,Q_2)\leq \int_{Q_2}|\nabla\phi_0|^q\mu\,dx\leq c\int_{Q_2}\mu\,dx
	\end{align*}
	with $c=c(p)$, and 
	\item for the last one we employed \cite[Lemma 3.5]{MenPhu12} to get
	\begin{align*}
	\dfrac{\mu(Q_2)}{\mu(Q_1)}\leq[\mu]_{\mathcal{A}_p}\left(\dfrac{|Q_2|}{|Q_1|}\right)^p\leq c(n,p,[\mu]_{\mathcal{A}_p}).
	\end{align*}
\end{itemize}
Then \eqref{eq:cap5} implies that for some $C$ satisfying \eqref{eq:cap} with some $c_2=c_2(n,p,[\mu]_{\mathcal{A}_p})$, \eqref{eq:wSP} holds when $r=1$. Now let us consider the case of $r\neq 1$. We use change of variables appropriately and apply \eqref{eq:cap5} to obtain 
	\begin{align}\label{eq:cap4.1}
		\begin{split}
			\left(\int_{Q_r}\left|u\right|^p\mu\,dx\right)^{\frac{1}{p}}&=\left(r^n\int_{Q_1}|u(ry)|^p\mu(ry)\,dy\right)^{\frac{1}{p}}\\
			&\leq cr^{\frac{n}{p}}\left(\dfrac{\left(\int_{Q_1}\mu(ry)\,dy\right)^{\frac{q}{p}}}{\text{Cap}_{q,\mu(r\cdot)}(\tilde{K},Q_2)}\int_{Q_1}|\nabla(u(ry))|^q\mu(ry)\,dy\right)^{\frac{1}{q}}\\
			&= cr^{\frac{n}{p}+1}\underbrace{\left[\dfrac{\left(\int_{Q_1}\mu(ry)\,dy\right)^{\frac{q}{p}}}{\text{Cap}_{q,\mu(r\cdot)}(\tilde{K},Q_2)}\right]^{\frac{1}{q}}}_{=:I_1}\underbrace{\left[\int_{Q_1}|\nabla u(ry)|^q\mu(ry)\,dy\right]^{\frac{1}{q}}}_{=:I_2}
		\end{split}
	\end{align}
	with some $c=c(n,p,[\mu]_{\mathcal{A}_p})$,	where we denote
	\begin{align*}
		\text{Cap}_{q,\mu(r\cdot)}(\tilde{K},Q_2):=\inf\left\{\int_{Q_{2}}|\nabla\phi|^q\mu(rx)\,dx:\phi\in C^{\infty}_0(Q_{2}),\,\,\phi\geq 1\,\,\text{on}\,\, \tilde{K}\right\}
	\end{align*}
	and $\tilde{K}:=\{x\in \bar{Q}_{1}:rx\in K\}$.

	For $I_1$, we see that
	\begin{align}\label{eq:cap7}
		\int_{Q_1}\mu(ry)\,dy\eqsim \dfrac{1}{r^n}\int_{Q_{r}}\mu(x)\,dx.
	\end{align}
	Moreover, for $\phi\in C^{\infty}(Q_2)$,
	\begin{align*}
		\int_{Q_2}|\nabla \phi|^q\mu(ry)\,dy\eqsim\int_{Q_{2r}}r^{q-n}\left|(\nabla\phi)\left(\dfrac{x}{r}\right)\right|^q\mu(x)\,dx
	\end{align*}
	and so we observe
	\begin{align}\label{eq:cap8}
		\text{Cap}_{q,\mu(r\cdot)}(\tilde{K},Q_2)\eqsim r^{q-n}\text{Cap}_{q,\mu}(K,Q_{2r}).
	\end{align}
	Hence it follows that
	\begin{align}\label{eq:cap9}
		I_1\eqsim r^{-1}\left[\dfrac{\left(\int_{Q_{r}}\mu(x)\,dx\right)^{\frac{q}{p}}}{\text{Cap}_{q,\mu}(K,Q_{2r})}\right]^{\frac{1}{q}}
	\end{align}
	with the implicit constant $c=c(n)$. On the other hand, $I_2$ can be computed as
	\begin{align}\label{eq:cap10}
		\begin{split}
			I_2=\left[\dfrac{1}{r^n}\int_{Q_r}|(\nabla u)(x)|^q\mu(x)\,dx\right]^{\frac{1}{q}}=r^{-\frac{n}{q}}\left[\int_{Q_r}|(\nabla u)(x)|^q\mu(x)\,dx\right]^{\frac{1}{q}}.
		\end{split}
	\end{align}
	Therefore, we obtain
	\begin{align}\label{eq:cap10.1}
		\left(\int_{Q_r}\left|u\right|^p\mu\,dx\right)^{\frac{1}{p}}\lesssim\left[\dfrac{\left(\int_{Q_{r}}\mu(x)\,dx\right)^{\frac{q}{p}}}{\text{Cap}_{q,\mu}(K,Q_{2r})}\right]^{\frac{1}{q}}\left[\int_{Q_r}|(\nabla u)(x)|^q\mu(x)\,dx\right]^{\frac{1}{q}}
	\end{align}
	with the implicit constant $c=c(n,p,[\mu]_{\mathcal{A}_p})$. Then \eqref{eq:wSP} follows for some $C$ such that \eqref{eq:cap} holds.

	\textbf{Step 3.} Now we will show (b) in case of $r=1$. Let $\psi\in C^{\infty}_0(Q_2)$ such that $\psi\geq 1$ on $K$ and 
	\begin{align}\label{eq:cap6}
		\norm{\nabla\psi}^q_{L^q(Q_2;\mu dx)}\leq\text{Cap}_{q,\mu}(K,Q_2)+\epsilon
	\end{align}
	for $\epsilon\in(0,1)$. Denoting $u=1-\psi$, we obtain
	\begin{align*}
		\int_{Q_1}|\nabla u|^q\mu\,dx\leq\int_{Q_2}| \nabla\psi|^q\mu\,dx.
	\end{align*}
	Then together with \eqref{eq:wSP'} and \eqref{eq:cap6}, it follows that
	\begin{align*}
		\int_{Q_{\frac{1}{2}}}|u|^q\mu\,dx\lesssim\,C\left(\text{Cap}_{q,\mu}(K,Q_2)+\epsilon\right).
	\end{align*}
	Now note that by \eqref{eq:doubling}, we see that
	\begin{align}\label{eq:cap10.2}
	\int_{Q_2}\mu\,dx\leq[\mu]_{\mathcal{A}_p}\left(\frac{|Q_2|}{|Q_{\frac{1}{2}}|}\right)^p\int_{Q_{\frac{1}{2}}}\mu\,dx\leq c(n,p,[\mu]_{\mathcal{A}_p})\int_{Q_{\frac{1}{2}}}\mu\,dx.
	\end{align}
	Two previous displays with H\"{o}lder's inequality yield
	\begin{align}\label{eq:cap11}
		\begin{split}
			1-\frac{1}{\int_{Q_{\frac{1}{2}}}\mu\,dx}\int_{Q_{\frac{1}{2}}}\psi\mu\,dx&=\frac{1}{\int_{Q_{\frac{1}{2}}}\mu\,dx}\int_{Q_{\frac{1}{2}}}u\mu\,dx\\
			&\lesssim\, C\left(\dfrac{\text{Cap}_{q,\mu}(K,Q_2)+\epsilon}{\int_{Q_{\frac{1}{2}}}\mu\,dx}\right)^{\frac{1}{q}}\\
			&\lesssim\, C\left(\dfrac{\text{Cap}_{q,\mu}(K,Q_2)+\epsilon}{\int_{Q_{2}}\mu\,dx}\right)^{\frac{1}{q}}.
		\end{split}
	\end{align}
	Now we want to show that $\frac{1}{\int_{Q_{1/2}}\mu\,dx}\int_{Q_{1/2}}\psi\mu\,dx$ is small. By \eqref{eq:SP0} of Lemma \ref{lem:SP0}, \eqref{eq:cap6} and \eqref{eq:cap10.2} again, we have
	\begin{align*}
		\frac{1}{\int_{Q_{\frac{1}{2}}}\mu\,dx}\int_{Q_{\frac{1}{2}}}\psi\mu\,dx&\leq \left(\frac{1}{\int_{Q_{\frac{1}{2}}}\mu\,dx}\int_{Q_{\frac{1}{2}}}\psi^q\mu\,dx\right)^{\frac{1}{q}}\\
		&\lesssim\left(\frac{1}{\int_{Q_{\frac{1}{2}}}\mu\,dx}\int_{Q_{\frac{1}{2}}}|\nabla\psi|^q\mu\,dx\right)^{\frac{1}{q}}\\
		&\leq c_0\left(\dfrac{\text{Cap}_{q,\mu}(K,Q_2)+\epsilon}{\int_{Q_{2}}\mu\,dx}\right)^{\frac{1}{q}}
	\end{align*}
	with $c_0=c_0(n,p,[\mu]_{\mathcal{A}_p})\geq 1$. Thus if $\epsilon\in(0,1)$ is sufficiently small, there holds
	\begin{align*}
		\frac{1}{\int_{Q_{\frac{1}{2}}}\mu\,dx}\int_{Q_{\frac{1}{2}}}\psi\mu\,dx\leq c_0\left(2\gamma \right)^{\frac{1}{q}}\leq c_0\left(2\gamma \right)^{\frac{1}{p}}
	\end{align*}
	since $\gamma\leq\frac{1}{2}$. Therefore, if $\gamma=\gamma(n,p,[\mu]_{\mathcal{A}_p})$ is sufficiently small so that $c_0\left(2\gamma \right)^{\frac{1}{p}}\leq 1/2$, from \eqref{eq:cap11} we find
	\begin{align*}
		C^{-q}\leq \dfrac{\text{Cap}_{q,\mu}(K,Q_2)+\epsilon}{\int_{Q_{2}}\mu\,dx}.
	\end{align*}
	Now since $\epsilon$ is arbitrary small, we have \eqref{eq:cap'} when $r=1$.

	In case of $r\neq 1$, by following the similar argument to \eqref{eq:cap4.1}--\eqref{eq:cap10.1}, we again obtain \eqref{eq:cap'}.
\end{proof}

\begin{remark} We have some remarks related to the above proposition.
\begin{enumerate}
	\item Since the set of functions $C^{\infty}(\overline{Q}_r)$ with $\text{dist}(\text{supp}\,u,K)>0$ is dense in $W^{1,q}(Q_{r},K;\mu dx)$, Proposition \ref{prop:SP} holds for any $u\in W^{1,q}(Q_{r},K;\mu dx)$.
	\item Proposition \ref{prop:SP} says that if $K\in \mathcal{N}_{\mu;\gamma,q}(Q_r)$ for small $\gamma\in(0,\frac{1}{2})$ depending on $n,p$ and $[\mu]_{\mathcal{A}_p}$, then
	\begin{align}\label{eq:cap.com}
	C\eqsim \dfrac{1}{r}\left(\dfrac{\mu(Q_{2r})}{\text{Cap}_{q,\mu}(K,Q_{2r})}\right)^{\frac{1}{q}}
	\end{align}
	holds with implicit constant $c=c(n,p,[\mu]_{\mathcal{A}_p})$, where $C$ is from \eqref{eq:wSP} and \eqref{eq:wSP'}. Thus in general, the above relation can be regarded as the optimal relation between the constant $C$ and $\text{Cap}_{q,\mu}(K,Q_{2r})$ for Sobolev-Poincar\'{e} inequality in Proposition \ref{prop:SP}.
	\item Due to \eqref{eq:cap7} and \eqref{eq:cap8}, one can see that the right-hand side of  \eqref{eq:cap.com} is scaling-invariant.
\end{enumerate}
\end{remark}

We end up this subsection with the following weighted Poincar\'{e} inequality, which is needed in the next subsection.

\begin{corollary}\label{cor:SP}
	Let $\Omega$ be $(L,R_0)$-Lipschitz for $L,R_0>0$, and let $\mathcal{D}\subset\partial\Omega$. Choose $Q_R\subset\setR^n$ such that $\Omega\subset Q_R$. If
	\begin{align}\label{eq:cap0}
	\textrm{Cap}_{p,\mu}(\mathcal{D},Q_{2R})> 0
	\end{align}
	holds, then for some $c=c(n,p,[\mu]_{\mathcal{A}_p},\mu(Q_{2R}),\text{Cap}_{p,\mu}(K,Q_{2R}),L,R_0)$ we have
	\begin{align}\label{eq:wSPOmega}
	\left(\int_{\Omega}\left|f\right|^p\mu\,dx\right)^{\frac{1}{p}}\leq c\left(\int_{\Omega}|\nabla f|^p\mu\,dx\right)^{\frac{1}{p}}
	\end{align}
	for any $f\in W^{1,p}(\Omega,\mathcal{D};\mu\,dx)$.
\end{corollary}
\begin{proof}
Since Lipschitz domain is an extension domain for given parameters $L>0$ and $R_0>0$, (see \cite{Chu92}), there exists $\overline{f}\in W^{1,p}(Q_{2R},\mathcal{D};\mu\,dx)$ such that $\overline{f}\equiv f$ in $\Omega$ with $\norm{\overline{f}}_{W^{1,p}(Q_{2R})}\lesssim \norm{f}_{W^{1,p}(\Omega)}$. Then \eqref{eq:wSPOmega} holds by Proposition \ref{prop:SP} (a).
\end{proof}

\subsection{Main settings, existence and uniqueness}
We are now ready to offer the rigorous settings of the problem in consideration. For $\Omega\subset\setR^n$ be a bounded $(L,R_0)$-Lipschitz domain. The coefficient function $\mathbb{A}:\setR^n\rightarrow\setR^{n\times n}_{\text{sym}}$ is positive definite matrix-valued weight with
\begin{align}
\label{eq:A}
\abs{\mathbb{A}(x)}\, \abs{\mathbb{A}^{-1}(x)}\leq\Lambda
\end{align} 
for any $x\in\setR^n$ with some constant $\Lambda\geq 1$, where by~$\abs{\cdot}$ we denote the spectral norm of a matrix, i.e., 
\begin{align*}
\abs{\mathbb{A}(x)}=\sup_{\abs{y}\leq 1}\dfrac{\abs{\mathbb{A}(x)y}}{\abs{y}}.
\end{align*}
Let us define
\begin{align*}
\omega(x):=\abs{\mathbb{A}(x)}.
\end{align*}
Then one can see that \eqref{eq:A} is equivalent to say that
\begin{align}\label{eq:equivalent}
\Lambda^{-1} \omega(x) \abs{\xi}^2 \leq \skp{\bbA(x) \xi}{\xi} \leq
\Lambda \omega(x) \abs{\xi}^2\quad(\xi\in\setR^n).
\end{align}
For $1<p<\infty$, assume $\omega\in \mathcal{A}_p(\setR^n)$ and let $u\in W^{1,p}(\Omega,\mathcal{D};\omega dx)$ be a weak solution of
\begin{align}\label{eq:eq0}
	\begin{split}
		-\divergence\left(\abs{\nabla u}^{p-2}\mathbb{A}(x)\nabla u\right)&=l \quad\text{in }\Omega,\\
		u&=0\quad\text{on }\mathcal{D},\\
		\dfrac{\partial u}{\partial \nu}&=0\quad\text{on }\mathcal{N},
	\end{split}
\end{align}
where $l$ is a linear functional on $W^{1,p}(\Omega,\mathcal{D};\omega dx)$, and $\frac{\partial u}{\partial \nu}=\sum^{n}_{i=1}\frac{\partial u}{\partial x_i}\nu_i$ is the outward conormal derivative of $u$. By Corollary \ref{cor:SP}, if we assume that
\begin{align*}
	\textrm{Cap}_{p,\omega}(\mathcal{D},Q_{2R})> 0,
\end{align*}
then the following weighted Poincar\'{e} inequality
\begin{align}\label{eq:wSPOmega0}
\left(\int_{\Omega}\left|u\right|^p\omega\,dx\right)^{\frac{1}{p}}\leq c\left(\int_{\Omega}|\nabla u|^p\omega\,dx\right)^{\frac{1}{p}}
\end{align}
holds for some $c>0$, so that the norm of $W^{1,p}(\Omega,\mathcal{D};\omega dx)$ is equivalent to the norm only associated with the gradient. Then the application of Hahn-Banach theorem yields that the functional $l$ is written as follows:
\begin{align*}
l(\phi)=-\sum^n_{i=1}\int_{\Omega}\mathbf{f}\cdot \nabla \phi\,dx
\end{align*}
for some $\mathbf{f}=(\mathbf{f}_1,\dots,\mathbf{f}_n):\Omega\rightarrow\setR^n$ with $|\mathbf{f}|$ belonging to the dual space of $L^{p}_{\omega}(\Omega)$, i.e., $|\mathbf{f}|\in L^{p'}(\Omega;\omega^{-\frac{1}{p-1}}\,dx)$.
By letting
\begin{align*}
G=(G_1,\dots,G_n)\quad\text{with}\quad G_i=\left|\mathbb{A}^{-1}(x)\mathbf{f}\right|^{\frac{2-p}{p-1}}\mathbb{A}^{-1}(x)\mathbf{f}_i,
\end{align*}
we have $|G|\in L^{p}(\Omega;\omega\,dx)$ and $\mathbf{f}(\cdot)=|G(\cdot)|^{p-2}\mathbb{A}(x)G(\cdot)$. Therefore, under the assumption \eqref{eq:wSPOmega}, \eqref{eq:eq0} is equivalent to
\begin{align}\label{eq:eq}
	\begin{split}
		-\divergence\left(\abs{\nabla u}^{p-2}\mathbb{A}(x)\nabla u\right)&=-\divergence\left(\abs{G}^{p-2}\mathbb{A}(x)G\right)\quad\text{in }\Omega,\\
		u&=0\quad\text{on }\mathcal{D},\\
		\dfrac{\partial u}{\partial \nu}&=0\quad\text{on }\mathcal{N}
	\end{split}
\end{align}
for $\abs{G}\in L^p(\Omega;\omega dx)$.
Note that since $\mathbb{A}(x)$ is symmetric and positive definite, $\mathbb{A}^{\frac{1}{2}}(x)$ is well-defined, still symmetric and positive definite. Then the above equation is the Euler-Lagrange equation of the following functional:
\begin{align*}
	v\mapsto \int_{\Omega}|\nabla v|^{p-2}|\mathbb{A}^{\frac{1}{2}}(x)\nabla v|^{2}\,dx-\int _{\Omega}\left<|G|^{p-2}\mathbb{A}^{\frac{1}{2}}(x)G,\mathbb{A}^{\frac{1}{2}}(x)Dv\right>\,dx
\end{align*}
with the admissible set $W^{1,p}(\Omega,\mathcal{D};\omega dx)$.

Now we are ready to give the existence and uniqueness of the weak solution of our problem with the help of the results in the previous section.
\begin{theorem}\label{thm:exist}
	Let $\Omega\subset\setR^n$ be a bounded $(L,R_0)$-Lipschitz domain. Assume that $\mathbb{A}(x):\setR^n\rightarrow\setR^{n\times n}_{\sym}$ satisfies $\abs{\mathbb{A}(x)}>0$ a.e., and $\abs{\mathbb{A}(x)}\abs{\mathbb{A}^{-1}(x)}\leq\Lambda$ for some $\Lambda>0$. Suppose that for $1<p<\infty$, $\abs{\mathbb{A}(x)}\in\mathcal{A}_p(\setR^n)$ and
	\begin{align}\label{eq:cap00}
	\textrm{Cap}_{p,\omega}(\mathcal{D},Q_{2R})> 0
	\end{align}
	hold for $\Omega\subset Q_R$ with $\omega(x)=\abs{\mathbb{A}(x)}$ and $\abs{G}\in L^p(\Omega;\omega dx)$. Then there exists a unique weak solution $u\in W^{1,p}(\Omega,\mathcal{D};\omega dx)$ of \eqref{eq:eq}. Moreover, we have the estimate
	\begin{align}\label{eq:en.est}
		\int_{\Omega}|\nabla u|^p\omega\,dx \leq c\int_{\Omega}|G|^p\omega\,dx
	\end{align}
	with $c=c(n,p,\Lambda)$.
\end{theorem}
\begin{proof}
By Corollary \ref{cor:SP}, \eqref{eq:cap00} implies \eqref{eq:wSPOmega0}. Then we obtain the existence by the standard theory of calculus of variations. The energy estimate \eqref{eq:en.est} follows from testing $u\in W^{1,p}(\Omega,\mathcal{D};\omega dx)$ to \eqref{eq:eq} and using \eqref{eq:equivalent} together with Young's inequality. Then from
\eqref{eq:equivalent} we have the monotonicity
\begin{align*}
\skp{\abs{\xi}^{p-2}\mathbb{A}(x)\xi-\abs{\eta}^{p-2}\mathbb{A}(x)\eta}{\xi-\eta}\gtrsim \omega(x)\abs{\xi-\eta}\quad(\xi,\eta\in\setR^n)
\end{align*}
and so the uniqueness follows in a standard way.
\end{proof}

We point out that the Poincar\'{e} inequality \eqref{eq:wSPOmega0} is used not only modifying the problem \eqref{eq:eq0} to \eqref{eq:eq}, but also used for the proof of existence of weak solution of \eqref{eq:eq}.

\begin{remark}
 The existence result could be obtained under the different assumption on the weight: ~$\omega, \omega^{-1}\in L^1(\Omega)$ using the results from Zhikov~\cite[Section 1.5]{Zhi11}. This situation is more difficult since smooth functions are in general not dense in~$W^{1,2}(\Omega,\omega dx)$ and presence of Lavrentiev phenomenon. However, these challenges do not impede the validity of the existence theorem.
\end{remark}

\subsection{Higher integrability of the gradient of the solution}
Here, we impose an additional assumption to prove higher integrability of the gradient of the solution. By \cite[Lemma 3.6]{MenPhu12}, there exists $q_2=q_2([\omega]_{\mathcal{A}_p})\in(1,p)$ such that $\omega\in \mathcal{A}_{q_2}$. Let us define
\begin{align*}
q:=
\begin{cases}
\max\left\{q_2,\frac{p+1}{2}\right\}&\quad\text{if }p\in\left(1,\frac{n}{n-1}\right],\\
\max\left\{q_2,\frac{np}{n+p}\right\}&\quad\text{if }p\in\left(\frac{n}{n-1},\infty\right).
\end{cases}
\end{align*}
Then we consider the following assumption.
\begin{assumption}\label{assump:1}
Let $r_0>0$ be given. for any $x_0\in \mathcal{D}$ and $r\in(0,r_0)$, we have
\begin{align}\label{eq:cen}
\dfrac{\textrm{Cap}_{q,\omega}(\mathcal{D}\cap Q_{r}(x_0),\setR^n)}{\omega(Q_r(x_0))}\geq c_0 r^{-q}
\end{align}
for some $c_0>0$ independent of $x_0$ and $r$.	
\end{assumption}

Now we give the following theorem.
\begin{theorem}[Higher integrability]\label{thm:hi} 
Let $\Omega\subset\setR^n$ be a bounded $(L,R_0)$-Lipschitz domain, $1<p<\infty$, $\omega\in\mathcal{A}_p(\Omega)$, and assume \eqref{eq:cen}. If $G\in L^{p+{\delta_0}}(\Omega;\omega dx)$ with $\delta_0>0$, then there exist $\delta_1=\delta_1(n,\delta_0)>0$ such that the following holds: the weak solution $u\in W^{1,p}(\Omega,\mathcal{D};\omega dx)$ of \eqref{eq:eq} satisfies $\abs{\nabla u}\in L^{p+\delta}(\Omega;\omega dx)$ for any  $\delta\leq\delta_1$. Moreover, the estimate
	\begin{align}\label{eq:hi}
		\int_{\Omega}|\nabla u|^{p+\delta}\omega(x)\,dx\leq c\int_{\Omega}|G|^{p+\delta}\omega(x)\,dx
	\end{align}
holds for some $c=c(n,p,\Lambda,[\omega]_{\mathcal{A}_p},q,L,R_0,|\Omega|,c_0,r_0)>0$.
\end{theorem}
\begin{proof}
We first choose $x_0\in\partial\Omega$. Then since $\Omega$ is $(L,R)$--Lipschitz, there exists a coordinate system $\{x_1,\dots,x_n\}$ and Lipschitz map $\psi:\setR^{n-1}\rightarrow\setR$ such that $x_0=0$ in this coordinate system, and there holds
\begin{align*}
\Omega\cap Q_R(x_0)=\{x=(x_1,\dots,x_n)=(x',x_n)\in Q_R(x_0):x_n>\psi(x')\}
\end{align*}
together with $\|\nabla\psi\|_{\infty}\leq L$. Let us write $Q_{R}=Q_{R}(x_0)$ and define $\Psi:\setR^n\rightarrow\setR^n$ as 
\begin{align*}
\Psi(x',x_n)=(x',x_n-\psi(x'))\,\,\text{for}\,\,(x',x_n)\in\setR^n.
\end{align*}
Then $\Psi$ is invertible, and we observe following facts from \cite{AlkCheMaz22}:
\begin{itemize}
	\item $\Psi(\partial\Omega\cap Q_{R})\subset T_{R}:=\{y:|y_i|<R\,\,\text{for}\,\,i=1,\dots,n-1,y_n=0\}$,
	\item $\widetilde{Q}_R:=Q_{(1+\sqrt{n-1}L)^{-1}R}\subset\Psi(Q_{R})\subset Q_{(1+\sqrt{n-1}L)R}$, and
	\item $\widetilde{Q}_R^+:=Q^+_{(1+\sqrt{n-1}L)^{-1}R}\subset\Psi(\Omega\cap Q_{R})$.
\end{itemize}
Now by transformation $y=\Psi(x)$, the problem \eqref{eq:eq} locally takes of the form
\begin{align}\label{eq:eq2}
\begin{split}
&-\divergence\left((\nabla\Psi(y))^t\widetilde{\mathbb{A}}(y)\abs{\nabla\Psi(y)\nabla \widetilde{u}(y)}^{p-2}(\nabla\Psi(y)\nabla \widetilde{u}(y))\right)\\
&\quad\quad\quad\quad\quad=-\divergence\left((\nabla\Psi(y))^t\widetilde{\mathbb{A}}(y)\abs{\widetilde{G}(y)}^{p-2}\widetilde{G}(y)\right)\quad\text{in }\widetilde{Q}^+_R
\end{split}
\end{align}
with
\begin{align*}
\begin{split}
\widetilde{u}=0\quad\text{on }\widetilde{\mathcal{D}}_R\quad\text{and}\quad
\dfrac{\partial \widetilde{u}}{\partial \widetilde{\nu}}=0\quad\text{on }\widetilde{\mathcal{N}}_{R}.
\end{split}
\end{align*}
Here, we list definitions of each quantities and its properties below:
\begin{itemize}
	\item We have $\widetilde{\mathbb{A}}(y):=(\mathbb{A}\circ\Psi^{-1})(y)$ and  $\widetilde{\omega}(y):=\abs{\widetilde{\mathbb{A}}(y)}\in \mathcal{A}_p$, since $\widetilde{Q}_R:=Q_{(1+\sqrt{n-1}L)^{-1}R}\subset\Psi(Q_{R})\subset Q_{(1+\sqrt{n-1}L)R}$ and $\psi$ is a Lipschitz map and so for any $Q_{R}$,
	\begin{align*}
	&\left(\dashint_{\widetilde{Q}_R}\widetilde{\omega}(y)\,dy\right)\left(\dashint_{\widetilde{Q}_R}\widetilde{\omega}(y)^{-\frac{1}{p-1}}\,dy\right)^{p-1}\\
	&\quad=\left(\dashint_{\Psi(Q_R)}\omega(x)|\nabla\Psi(x)|\,dx\right)\left(\dashint_{\Psi(Q_R)}\omega(x)^{-\frac{1}{p-1}}|\nabla\Psi(x)|\,dx\right)^{p-1}\\
	&\quad\lesssim \left(\dashint_{Q_{(1+\sqrt{n-1}L)R}}\omega(x)\,dx\right)\left(\dashint_{Q_{(1+\sqrt{n-1}L)R}}\omega(x)^{-\frac{1}{p-1}}\,dx\right)^{p-1}\leq [\omega]_{\mathcal{A}_p}
	\end{align*}
	with the implicit constant $c=c(n,L)>0$.
	\item We denote $\widetilde{\mathcal{D}}_{R}=\Psi(\mathcal{D}\cap Q_R)\cap \widetilde{Q}_{R}$ and $\widetilde{\mathcal{N}}_{R}=\Psi(\mathcal{N}\cap Q_{R})\cap \widetilde{Q}_{R}$.
	\item Note that $\widetilde{u}(y):=(u\circ\Psi^{-1})(y)\in W^{1,p}(\widetilde{Q}^+_{R},\widetilde{\mathcal{D}}_R;\widetilde{\omega} dy)$, and $\widetilde{G}(y):=(G\circ\Psi^{-1})(y)\in L^p(\widetilde{Q}^+_{R};\widetilde{\omega} dy)$ by the fact that $\psi$ is a Lipschitz map.
	\item The outer conormal derivative of $\widetilde{u}$ is denoted by $\frac{\partial\widetilde{u}}{\partial{\tilde{\nu}}}$, and since $\Psi(\partial\Omega\cap Q_{R})\subset T_{R}$, we have $\frac{\partial\widetilde{u}}{\partial{\tilde{\nu}}}=\partial_{n}\widetilde{u}$.
\end{itemize}

We use the even extension for the weight $\widetilde{\omega}(y)$ and $\widetilde{u}$ with respect to $\{y:y_n=0\}$. Let us write the extended function of $\widetilde{\omega}$ as $\widehat{\omega}$, and $\widetilde{u}$ as $\widehat{u}$. Then $\widehat{\omega}(y)\in\mathcal{A}_p$ and $\widehat{u}\in W^{1,p}(\widetilde{Q}_{R},\widetilde{\mathcal{D}}_R;\widehat{\omega} dy)$ satisfies the following equation:
\begin{align}\label{eq:eq3}
\begin{split}
-\divergence\left(\abs{\widehat{\nabla} \widehat{u}(y)}^{p-2}\widehat{\mathbb{A}}(y)\nabla \widehat{u}(y)\right)=-\divergence\left(\widehat{G}(y)\right)\quad\text{in }\widetilde{Q}_R
\end{split}
\end{align}
with
\begin{align*}
\begin{split}
\widehat{u}=0\quad\text{on }\widetilde{\mathcal{D}}_R.
\end{split}
\end{align*}
Here, let us list definitions of each quantities and properties; all of the extensions mentioned below are with respect to $\{y:y_n=0\}$.
\begin{itemize}
	\item We define $\widehat{\nabla} \widehat{u}(y)$ as an odd extension of $\nabla\Psi(y)\nabla \widetilde{u}(y)$, and note that since $\psi$ is Lipschitz, $|\widehat{\nabla} \widehat{u}(y)|\eqsim |\nabla \widetilde{u}(y)|$ holds.
	\item The coefficient matrix $A(y):=(\nabla\Psi(y))^t\widetilde{\mathbb{A}}(y)(\nabla\Psi(y))$ is such that $A_{in}(y)=A_{ni}(y)$ for $i\neq n$ are odd extensions of $[(\nabla\Psi(y))^t\widetilde{\mathbb{A}}(y)(\nabla\Psi(y))]_{in}$, and all other elements of $A_{ij}(y)$ are even extensions of $[(\nabla\Psi(y))^t\widetilde{\mathbb{A}}(y)(\nabla\Psi(y))]_{ij}$. Note that $A(y)$ is uniformly elliptic.
	\item We define $\widehat{G}(y)=(\widehat{G}_1(y),\dots,\widehat{G}_n(y))\in L^p(\widetilde{Q}^+_{R};\widetilde{\omega} dy)$ as such that $\widehat{G}_i(y)$ are even extensions of $((\nabla\Psi(y))^t|\widetilde{G}(y)|^{p-2}\widetilde{\mathbb{A}}(x)\widetilde{G}(y))_i$ for $i\neq n$, and $\widehat{G}_n(y)$ are the odd extension of $((\nabla\Psi(y))^t|\widetilde{G}(y)|^{p-2}\widetilde{\mathbb{A}}(x)\widetilde{G}(y))_n$. Note that 
	\begin{align*}
	|\widehat{G}|\eqsim \omega|\widetilde{G}|^{p-1}.
	\end{align*}
\end{itemize}
Then one can see that $\widehat{u}\in W^{1,p}(\widetilde{Q}_{R},\widetilde{\mathcal{D}}_R;\widehat{\omega} dy)$ from \eqref{eq:eq3} satisfies the following
\begin{align}\label{eq:eq5}
\int_{\widetilde{Q}_{R}}\widehat{\omega}(y)A(y)\abs{\widehat{\nabla} \widehat{u}}^{p-2}\nabla \widehat{u}\cdot\nabla \phi\,dy=\int_{\widetilde{Q}_{R}}\widehat{\omega}(y)\widehat{G}\cdot \nabla \phi\,dy
\end{align}
for any test functions $\phi\in W^{1,p}_0(\widetilde{Q}_{R},\widetilde{\mathcal{D}}_R;\widehat{\omega} dy)$.	
	
From now on, we suppose that
	\begin{align*}
		y_0\in \widetilde{Q}_{R/2}\setminus\partial \widetilde{Q}_{R/2},\quad\text{and }r\leq\frac{1}{2}\distance(y_0,\partial \widetilde{Q}_{R/2}).
	\end{align*}
First, consider the case when $Q_{3r/2}(y_0)\subset \widetilde{Q}_{R}\setminus \widetilde{\mathcal{D}}_{R}$ and test $\phi=(\widehat{u}-\lambda)\eta^p$ in \eqref{eq:eq5}, where $\lambda=\mean{\widehat{u}}_{Q_{3r/2}(y_0)}$, and the cutoff function $\eta\in C^{\infty}_0(Q_{3R/2}(y_0))$ satisfies $0\leq\eta\leq 1$, $\eta\equiv 1$ in $Q_{r}(y_0)$ and $|\nabla \eta|\leq c/r$. As a result, by \eqref{eq:eq5} and Young's inequality together with $|\widehat{G}|\eqsim |\widetilde{G}|^{p-1}$, we arrive at the Caccioppoli inequality
	\begin{align}\label{eq:hi3}
		\int_{Q_{r}(y_0)}|\nabla \widehat{u}|^p\widehat{\omega} \,dy\leq c(n,p,\Lambda)\left(\int_{Q_{3r/2}(y_0)}\left|\frac{\widehat{u}-\lambda}{r}\right|^p\widehat{\omega}\,dy+\int_{Q_{3r/2}(y_0)}|\widehat{G}|^p\widehat{\omega}\,dy\right).
	\end{align}
	Further, from Lemma \ref{lem:SP0}, together with $\widehat{\omega}(Q_r)\eqsim \widehat{\omega}(Q_{3r/2})$ from \eqref{eq:UB} and \eqref{eq:BU}, we have
	\begin{align*}
		\left(\frac{1}{\widehat{\omega}(Q_r)}\int_{Q_{3r/2}(y_0)}\left|\frac{\widehat{u}-\lambda}{r}\right|^{p}\widehat{\omega}\,dx\right)^{1/{p}}\leq c(n)\left(\frac{1}{\widehat{\omega}(Q_r)}\int_{Q_{3r/2}(y_0)}|\nabla \widehat{u}|^{q}\widehat{\omega}\,dx\right)^{1/{q}}.
	\end{align*}
	The above two inequalities yield that
	\begin{align}\label{eq:hi3.1}
		\begin{split}
			&\left(\frac{1}{\widehat{\omega}(Q_r)}\int_{Q_r(y_0)}|\nabla \widehat{u}|^p\widehat{\omega}\,dy\right)^{1/p}\\
			&\quad\leq c\left(\left(\frac{1}{\widehat{\omega}(Q_r)}\int_{Q_{2r}(y_0)}|\nabla \widehat{u}|^{q}\widehat{\omega}\,dy\right)^{1/q}+\left(\frac{1}{\widehat{\omega}(Q_r)}\int_{Q_{2r}(y_0)}|\widehat{G}|^p\widehat{\omega}\,dy\right)^{1/p}\right).
		\end{split}
	\end{align}

	Next we consider the case when $Q_{3r/2}(y_0)\cap \widetilde{\mathcal{D}}_{R}\neq \emptyset$. Choosing a test function $\phi=u\eta^p$ in \eqref{eq:eq5}, we find the similar estimate to \eqref{eq:hi3} as follows:
	\begin{align}\label{eq:hi4}
		\int_{Q_r(y_0)}|\nabla \widehat{u}|^p\widehat{\omega}\,dy\leq c(n,p,\Lambda)\left(\int_{Q_{2r}(y_0)}\left|\frac{\widehat{u}}{r}\right|^p\widehat{\omega}\,dy+\int_{Q_{2r}(y_0)}|\widehat{G}|^p\widehat{\omega}\,dy\right).
	\end{align}
	Now we apply the Sobolev-Poincar\'{e} inequality (Proposition \ref{prop:SP}) with the assumption \eqref{eq:cen}. Since $Q_{3r/2}(y_0)\cap\widetilde{\mathcal{D}}_{R}\neq\emptyset$, there exists a point $z_0\in Q_{3r/2}(y_0)\cap\widetilde{\mathcal{D}}_{R}$ such that $\overline{Q}_{r/2}(z_0)\subset\overline{Q}_{2r}(y_0)$. Let $z=\Psi^{-1}(z_0)\in \mathcal{D}\cap Q_{R}$. Note that there is a cube $Q_{cr}(z)$ with some $c=c(L,n)>0$ such that $Q_{cr}(z)\subset\Psi^{-1}(\overline{Q}_{r/2}(z_0))$. By \eqref{eq:cen},
	\begin{align}\label{eq:cen1}
		\dfrac{1}{\int_{Q_{cr}(z)}\omega\,dx}\textrm{Cap}_{q,\omega}(\mathcal{D}\cap Q_{cr}(z),\setR^n)\geq c r^{-q}
	\end{align}
	holds, where $c=c(n,L,c_0)$. Then considering the transformation with respect to $\Psi$, it follows that
	\begin{align}\label{eq:cen2}
		\dfrac{1}{\int_{Q_{2r}}\omega\,dx}\textrm{Cap}_{q,\omega}(\widetilde{\mathcal{D}}_{R},\setR^n)\geq c r^{-q}.
	\end{align}
	Hence we obtain
	\begin{align*}
		\left(\frac{1}{\widehat{\omega}(Q_r)}\int_{Q_{2r}(y_0)}\left|\frac{\widehat{u}}{r}\right|^p\widehat{\omega}\,dy\right)^{\frac{1}{p}}\leq c\left(\frac{1}{\widehat{\omega}(Q_r)}\int_{Q_{2r}(y_0)}|\nabla \widehat{u}|^{q}\widehat{\omega}\,dy\right)^{\frac{1}{q}}.
	\end{align*}

	Combining with \eqref{eq:hi4}, we again obtain \eqref{eq:hi3.1}. Now we use Gehring's lemma \cite[Theorem 1]{MasMil00} with the weight $\widehat{\omega}$ to obtain
	\begin{align*}
		\begin{split}
			&\left(\frac{1}{\widehat{\omega}(Q_r)}\int_{Q_r(y_0)}|\nabla \widehat{u}|^{p+\delta}\widehat{\omega}\,dy\right)^{1/(p+\delta)}\\
			&\quad\leq c\left(\left(\frac{1}{\widehat{\omega}(Q_r)}\int_{Q_{2r}(y_0)}|\nabla \widehat{u}|^{p}\widehat{\omega}\,dy\right)^{1/p}+\left(\frac{1}{\widehat{\omega}(Q_r)}\int_{Q_{2r}(y_0)}|\widehat{G}|^{p+\delta}\widehat{\omega}\,dy\right)^{1/(p+\delta)}\right).
		\end{split}
	\end{align*}
	Multiplying $\widehat{\omega}(Q_r)^{1/p}$ to both sides of the above estimate, we find that
\begin{align*}
\begin{split}
&\left(\int_{Q_r(y_0)}|\nabla \widehat{u}|^{p+\delta}\widehat{\omega}\,dy\right)^{1/(p+\delta)}\\
&\quad\leq\frac{c}{\widehat{\omega}(Q_r)^{1/(p+\delta)-1/p}}\left(\int_{Q_{2r}(y_0)}|\nabla \widehat{u}|^{p}\widehat{\omega}\,dy\right)^{1/p}+c\left(\int_{Q_{2r}(y_0)}|\widehat{G}|^{p+\delta}\widehat{\omega}\,dy\right)^{1/(p+\delta)}.
\end{split}
\end{align*}
Since the function $\widehat{u}$ was the even extension of $\widetilde{u}$ with respect to $\{y:y_n=0\}$, together with using the construction of $\widehat{\omega}$ and $\widehat{G}$ from $\widetilde{\omega}$ and $\widetilde{G}$ respectively, the above estimate implies
\begin{align*}
\begin{split}
&\left(\int_{Q_r(y_0)}|\nabla \widetilde{u}|^{p+\delta}\widetilde{\omega}\,dy\right)^{1/(p+\delta)}\\
&\quad\leq c\left(\int_{Q_{2r}(y_0)}|\nabla \widetilde{u}|^{p}\widetilde{\omega}\,dy\right)^{1/p}+c\left(\int_{Q_{2r}(y_0)}|\widetilde{G}|^{p+\delta}\widetilde{\omega}\,dy\right)^{1/(p+\delta)}
\end{split}
\end{align*}
with $c=c(n,p,\Lambda,[\omega]_{\mathcal{A}_p},c_0,L,R_0)>0$. Now using the covering argument together with the following interior estimate
	\begin{align*}
		\begin{split}
			&\left(\int_{Q_r}|\nabla \widetilde{u}|^{p+\delta}\widetilde{\omega}\,dy\right)^{1/(p+\delta)}\\
			&\quad\leq c\left(\int_{Q_{2r}}|\nabla \widetilde{u}|^{p}\widetilde{\omega}\,dy\right)^{1/p}+c\left(\int_{Q_{2r}}|\widetilde{G}|^{p+\delta}\widetilde{\omega}\,dy\right)^{1/(p+\delta)}
		\end{split}
	\end{align*}
	with $Q_{2r}\Subset\Omega$, we obtain the conclusion.
\end{proof}

\begin{remark}
	Inspired by the example from~\cite{AlkCheMaz22} for uniformly elliptic weights, we give the example of the set $F$ which has zero $1$-dimensional measure when $n=2$, but satisfying \eqref{eq:cen} in the case of $p=2$. 
	We choose $\omega(x)=\abs{x}^{1/2}\in\mathcal{A}_2(Q_1)$, $q=3/2$ and $q_0=13/10$. Then we have
	\begin{itemize}
		\item $1<q_0<q$, and
		\item $\omega\in\mathcal{A}_q$ and $\omega\in\mathcal{A}_{q_0}$.
	\end{itemize}

	We begin with the one dimensional generalized Cantor
set~$\frC_\lambda$ with $\lambda \in (0,\frac 12)$, which is also
known as the (1-2$\lambda$)-middle Cantor set. 
\begin{enumerate}
	\item We start with the
	interval~$\frC_{\lambda,0} := (-\frac 12,\frac 12)$.
	\item We
	define~$\frC_{\lambda,k+1}$ inductively by removing the
	middle~$1-2\lambda$ parts from~$\frC_{\lambda,k}$.
	\item In particular, we
	define~$\frC_\lambda := \cap_{k \geq 1} \frC_{\lambda,k}$.
\end{enumerate}
Then the fractal (Hausdorff) dimension
of~$\frC_\lambda$ is
$\frD=\dim(\frC_\lambda) = \log(2)/ \log(1/\lambda) \in (0,1)$,
i.e. $\lambda = 2^{-\frD}$. The set $\frC_\lambda$ has clearly zero 1-dimensional measure for any $\lambda\in(0,\frac{1}{2})$.
Now define
\begin{itemize}
	\item $\lambda\in(0,1/2)$ be such that 
	\begin{align*}
	\dfrac{\log2}{-\log\lambda}>2-\dfrac{q}{q_0},\quad\text{i.e.,}\quad \lambda>2^{-\frac{13}{11}},
	\end{align*}
	\item $l_i=\lambda l_{i-1}$ for $i\in\setN$, and
	\item $F=\frC_\lambda$.
\end{itemize}
Then Since the Hausdorff dimension of $F$ satisfies $\mathcal{H}(F)>2-\frac{q}{q_0}$, by \cite[Corollary 2.33]{HeiKilMar06} we have 	
	\begin{align}\label{eq:cap>0}
	Cap_{q,\omega}(F,\setR^2)>0.
	\end{align}

	Now it remains to show that \eqref{eq:cen} holds. Fix $x_0\in F$ and $r\in(0,r_0)$. Recall the definition of the capacity of the set $F\cap Q_r(x_0)$:
	\begin{align}\label{eq:capq}
	\text{Cap}_{q,\omega}(F\cap Q_r(x_0),\setR^2)=\inf\left\{\int_{\setR^2}|\nabla \phi|^{q}\omega\,dx:\phi\in C^{\infty}_0(\setR^2),\phi\geq 1\,\,\text{on}\, F\cap Q_r(x_0)\right\}.
	\end{align}
	If $r\leq r_0\leq 1/3$, then there is a natural number $l_0$ such that $3^{-l_0-1}<r\leq 3^{-l_0}$. Since $F\cap [x_0-l_0,x_0+l_0]\subset F\cap Q_r(x_0)$, there holds
	\begin{align}\label{eq:capq2}
	\text{Cap}_{q,\omega}(F\cap Q_r(x_0),\setR^2)\geq \text{Cap}_{q,\omega}(F\cap [x_0-l_0,x_0+l_0],\setR^2).
	\end{align}
	Using \eqref{eq:capq} and the affine transformation
	\begin{align}\label{eq:trans}
	y=(x-x_0)/r+x_0,\,\,\text{where }r=3^{-l_0-2},
	\end{align}
	and employing \eqref{eq:capq2} together with the scaling property \eqref{eq:cap8}, we obtain
	\begin{align}\label{eq:capq3}
	\begin{split}
	\text{Cap}_{q,\omega}(F\cap Q_r(x_0),\setR^2)&\geq 3^{-(l_0+2)/2}\text{Cap}_{q,\omega}(\tilde{F},\setR^2)\\
	&\geq 3^{-1/2}r^{1/2}\text{Cap}_{q,\omega}(\tilde{F},\setR^2)\geq 3^{-1/2}r\text{Cap}_{q,\omega}(\tilde{F},\setR^2),
	\end{split}
	\end{align}
	where $\tilde{F}$ means the image of the set $[x_0-l_0,x_0+l_0]\cap F$ by the transformation \eqref{eq:trans}. We note that under \eqref{eq:trans}, the set $\tilde{F}$ is a shift of the Cantor set $F$ along the $x_1$-axis. Note that 
	\begin{align*}
	\int_{Q_r(x_0)}\omega(x)\,dx=\int_{Q_r(x_0)}|x|^{1/2}\,dx\eqsim r^{5/2}
	\end{align*}
	uniformly on $x_0\in F\subset(-\frac{1}{2},\frac{1}{2})$. Then by \eqref{eq:cap>0} and \eqref{eq:capq3}, \eqref{eq:cen} holds for constants $q=3/2$, $c_0\eqsim \text{Cap}_{q,\omega}(\tilde{F},\setR^2)>0$ and $r_0=1/3$.
\end{remark}

\section{Zaremba problem with degenerate weights as multipliers}\label{sec:4}

In this section we establish Zaremba problem with degenerate weights where the weights are multipliers of the gradient of the solution. First of all, we show an analogous version of Proposition \ref{prop:SP}.

\begin{proposition}\label{prop:SP'}
	Let $K$ be a closed subset of $\overline{Q}_r$ and $\mu^p\in\mathcal{A}_p(\setR^n)$.
	\begin{enumerate}
		\item There exists $p_0=p_0([\mu]_{\mathcal{A}_p})$ such that for each $q\in[p_0,p]$, 
		\begin{align}\label{eq:wSP''}
		\left(\dashint_{Q_r}\left|\dfrac{u}{r}\right|^p\mu^p\,dx\right)^{\frac{1}{p}}\leq C\left(\dashint_{Q_r}|\nabla u|^q\mu^q\,dx\right)^{\frac{1}{q}}
		\end{align}
		holds for all $u\in C^{\infty}(\bar{Q}_r)$ with $\text{dist}(\text{supp}\,u,K)>0$. Moreover, $C$ satisfies
		\begin{align}\label{eq:cap''}
		C\leq \frac{c_2}{r}\left(\dfrac{\mu^q(Q_{2r})}{\text{Cap}_{q,\mu^q}(K,Q_{2r})}\right)^{\frac{1}{q}}
		\end{align}
		with $c_4=c_4(n,p,[\mu^p]_{\mathcal{A}_p})$.
		\item For any $u\in C^{\infty}(\bar{Q}_r)$ with $\text{dist}(\text{supp}\,u,K)>0$, let us assume that
		\begin{align}\label{eq:wSP'''}
		\left(\dashint_{Q_{\frac{r}{2}}}\left|\frac{u}{r}\right|^p\mu^p\,dx\right)^{\frac{1}{p}}\leq C\left(\dashint_{Q_r}|\nabla u|^q\mu^q\,dx\right)^{\frac{1}{q}}
		\end{align}
		holds for some $q\in[p_0,p]$. Then there exists a small $\gamma=\gamma(n,p,[\mu^p]_{\mathcal{A}_p})$ such that if $K\in \mathcal{N}_{\mu^q;\gamma,q}(Q_r)$, we have
		\begin{align}\label{eq:cap'''}
		C\geq \dfrac{c_5}{r}\left(\dfrac{\mu^q(Q_{2r})}{\text{Cap}_{q,\mu^q}(K,Q_{2r})}\right)^{\frac{1}{q}}
		\end{align}
		with $c_5=c_5(n,p,[\mu^p]_{\mathcal{A}_p})$.
	\end{enumerate}
\end{proposition}

\begin{proof}
	We follow the proof of the above lemma with slight modifications. The proof is proceeded in the case of $r=1$ and then use the scaling. We start the proof by noting that there exists 
	\begin{align}\label{eq:gamma1}
	\gamma_1=\gamma_1(n,p,[\mu^p]_{\mathcal{A}_p})\in(\tfrac{np}{n+p},p)
	\end{align}
	such that $\mu^p\in\mathcal{A}_p$ implies $\mu^{\gamma}\in\mathcal{A}_{\gamma}$ for any $\gamma\in[\gamma_1,p]$. Moreover, $[\mu^{\gamma}]_{\mathcal{A}_{\gamma}}\leq c(n,p,[\mu^p]_{\mathcal{A}_p})$. For its proof, use \cite[Lemma 3.6]{MenPhu12} together with the reverse H\"{o}lder inequality of the weight $\mu^p$ as in \cite[Theorem 7.2.2]{Gra14class}.

	Similar to the proof of \textbf{Step 1} in Proposition \ref{prop:SP}, we see that if $K\in \overline{Q}_1$ is a compact set, then
	\begin{align}\label{eq:cap1'}
	\begin{split}
	&\text{Cap}_{q,\mu^q}(K,Q_2)\\
	&\quad\eqsim\inf\left\{\sum^{1}_{m=0}\norm{\nabla_m(1-u)}^q_{L^q(Q_1;\mu^q dx)}:u\in C^{\infty}(\overline{Q}_1),\text{dist}(\text{supp}\,u,K)>0\right\}
	\end{split}
	\end{align}
	for any $q\in[q_0,p]$ with the implicit constant $c=c(n,p,[\mu^p]_{\mathcal{A}_p})$. The proof is overall similar to the one in Proposition \ref{prop:SP} except using \eqref{eq:SP0'} of Lemma \ref{lem:SP0'}, instead of \eqref{eq:SP0} of Lemma \ref{lem:SP0}.

	Now we will show (a) when $r=1$.
	Let $u\in C^{\infty}(\bar{Q}_1)$, $\text{dist}(\text{supp}\,u,K)>0$ and the number $N$ be such that
	\begin{align*}
	N^q=\dfrac{1}{\int_{Q_1}\mu^q\,dx}\int_{Q_1}|u|^q\mu^q\,dx.
	\end{align*}
	Again following the proof of \textbf{Step 2} with the weighted Poincar\'{e} and Sobolev-Poincar\'{e} inequality in Lemma \ref{lem:SP0'}, we show \eqref{eq:cap''} in case of $r=1$. If $r\neq 1$, the similar argument of change of variables from \eqref{eq:cap4.1} to \eqref{eq:cap10.1} works. Then (a) is proved.
	
	The statement (b) is proved similarly to above by using the weighted Poincar\'{e} inequality Lemma \ref{lem:SP0'}, together with \eqref{eq:doubling}. 
\end{proof}

The following Poincar\'{e} inequality for general Lipschitz domain, which is needed to main settings to our problem, also holds. The proof is as same as Corollary \ref{cor:SP} using Proposition \ref{prop:SP'} instead of Proposition \ref{prop:SP}.

\begin{corollary}\label{cor:SP'}
	Let $\Omega$ be $(L,R_0)$-Lipschitz for $L,R_0>0$, and let $\mathcal{D}\subset\partial\Omega$. Choose $Q_R\subset\setR^n$ such that $\Omega\subset Q_R$. If
	\begin{align}\label{eq:cap0'}
	\textrm{Cap}_{p,\mu^p}(\mathcal{D},Q_{2R})> 0
	\end{align}
	holds, then for some $c=c(n,p,[\mu^p]_{\mathcal{A}_p},\mu^p(Q_{2R}),\text{Cap}_{q,\mu^p}(K,Q_{2R}),L,R_0)$ we have
	\begin{align}\label{eq:wSPOmega'}
	\left(\int_{\Omega}\left|f\right|^p\mu^p\,dx\right)^{\frac{1}{p}}\leq c\left(\int_{\Omega}|\nabla f|^p\mu^p\,dx\right)^{\frac{1}{p}}
	\end{align}
	for any $f\in W^{1,p}(\Omega,\mathcal{D};\mu^p\,dx)$.
\end{corollary}

Using the above corollary, we can prove the existence and uniqueness result for \eqref{eq:plap.multiplier}: let $\setM:\setR^n\rightarrow\setR^{n\times n}_{\sym}$ be a matrix valued weight, i.e., $\abs{\setM(x)}>0$ a.e. for every $x\in\setR^n$ with $\abs{\cdot}$ being spectral norm of a matrix. Let us assume that $\abs{\setM(\cdot)}^p\in\mathcal{A}_p(\setR^n)$, $\setM(x)$ is symmetric and
\begin{align*}
\abs{\setM(x)}\abs{\setM(x)^{-1}}\leq \Lambda
\end{align*}
for some $\Lambda>0$. Define $\abs{\setM(x)}=\widetilde{\omega}(x)$. Then by \cite{BalDieGioPas22}, we have
\begin{align}\label{eq:ellip}
\Lambda^{-1}\widetilde{\omega}(x)|\xi|\leq |\setM(x)\xi|\leq \widetilde{\omega}(x)|\xi|\quad(\xi\in\setR^n).
\end{align}
Then we consider \eqref{eq:plap.multiplier} for $1<p<\infty$, where $l$ is a linear functional on $W^{1,p}(\Omega,\mathcal{D};\widetilde{\omega}^p dx)$, $\widetilde{\omega}^p\in\mathcal{A}_p(\setR^n)$, and $\frac{\partial u}{\partial \overline{\nu}}=\sum^{n}_{i=1}\setM_{ij}(x)\frac{\partial u}{\partial x_j}\nu_i$ is the outward conormal derivative of $u$. By Corollary \ref{cor:SP'}, if we assume that
\begin{align*}
\textrm{Cap}_{p,\widetilde{\omega}^p}(\mathcal{D},Q_{2R})> 0,
\end{align*}
then the weighted Poincar\'{e} inequality
\begin{align}\label{eq:wSPOmega0'}
\left(\int_{\Omega}\left|u\right|^p\widetilde{\omega}^p\,dx\right)^{\frac{1}{p}}\leq c\left(\int_{\Omega}|\nabla u|^p\widetilde{\omega}^p\,dx\right)^{\frac{1}{p}}
\end{align}
holds for some $c>0$, and so the norm of $W^{1,p}(\Omega,\mathcal{D};\widetilde{\omega} dx)$ is equivalent to the norm only with the gradient. Then Hahn-Banach theorem helps us to find that the functional $l$ is written as follows:
\begin{align*}
l(\phi)=-\sum^n_{i=1}\int_{\Omega}\widetilde{\mathbf{f}}\cdot \nabla \phi\,dx,
\end{align*}
where $\widetilde{\mathbf{f}}=(\widetilde{\mathbf{f}}_1,\dots,\widetilde{\mathbf{f}}_n):\Omega\rightarrow\setR^n$ with $|\widetilde{\mathbf{f}}|$ belonging to the dual space of $L^{p}_{\widetilde{\omega}^p}(\Omega)$, i.e., $|\widetilde{\mathbf{f}}|\in L^{p'}(\Omega;\widetilde{\omega}^{-p'}\,dx)$.
By letting
\begin{align*}
\widetilde{G}=(\widetilde{G}_1,\dots,\widetilde{G}_n)\quad\text{with}\quad \widetilde{G}_i=\left|\setM^{-1}\widetilde{\mathbf{f}}\right|^{\frac{2-p}{p-1}}\setM^{-1}\widetilde{\mathbf{f}}_i,
\end{align*}
we have $|\widetilde{G}|\in L^{p}(\Omega;\widetilde{\omega}^p\,dx)$ and $\widetilde{\mathbf{f}}(\cdot)=|\setM(\cdot)\widetilde{G}(\cdot)|^{p-2}\setM^2(\cdot)G(\cdot)$. Therefore, under the assumption \eqref{eq:wSPOmega'}, \eqref{eq:plap.multiplier} is equivalent to the following: For $\widetilde{G}\in L^p(\Omega;\widetilde{\omega}^p dx)$, let $u\in W^{1,p}(\Omega,\mathcal{D};\widetilde{\omega}^p dx)$ be a weak solution of 
\begin{align}\label{eq:eqw}
\begin{split}
-\divergence\left(\setM(x)^2\abs{\setM(x)\nabla u}^{p-2}\nabla u\right)&=-\divergence\left(\setM(x)^2\abs{\setM(x)\widetilde{G}}^{p-2}\widetilde{G}\right)\quad\text{in }\Omega,\\
u&=0\quad\text{on }\mathcal{D},\\
\dfrac{\partial u}{\partial \overline{\nu}}&=0\quad\text{on }\mathcal{N},
\end{split}
\end{align}
where $\frac{\partial u}{\partial \overline{\nu}}=\sum^{n}_{i=1}\setM_{ij}(x)\frac{\partial u}{\partial x_j}\nu_i$ is the outward conormal derivative of $u$ corresponding this problem.

Note that the above equation is the Euler-Lagrange equation of the following functional:
\begin{align*}
W^{1,p}(\Omega,\mathcal{D};\widetilde{\omega}^p dx)\ni v\mapsto \int_{\Omega}|\setM(x)\nabla v|^{p}\,dx-\int _{\Omega}\left<|\setM(x)G|^{p-2}\setM(x)G,\setM(x)Dv\right>\,dx.
\end{align*}
Then we also have the following existence and uniqueness result to the problem \eqref{eq:eqw}. The proof is similar to Theorem \ref{thm:exist}.
\begin{theorem}
Let $\Omega\subset\setR^n$ be a bounded $(L,R_0)$-Lipschitz domain. Assume that $\mathbb{M}(x):\setR^n\rightarrow\setR^{n\times n}_{\sym}$ satisfies $\abs{\mathbb{M}(x)}>0$ a.e., and $\abs{\mathbb{M}(x)}\abs{\mathbb{M}^{-1}(x)}\leq\Lambda$ for some $\Lambda>0$. Suppose that for $1<p<\infty$, $\abs{\mathbb{M}(x)}^p\in\mathcal{A}_p(\setR^n)$ and
\begin{align}\label{eq:cap0w}
\textrm{Cap}_{p,\widetilde{\omega}^p}(\mathcal{D},Q_{2R})> 0
\end{align}
hold for $\Omega\subset Q_R$ with $\widetilde{\omega}(x)=\abs{\mathbb{M}(x)}$ and $\abs{G}\in L^p(\Omega;\widetilde{\omega}^p dx)$. Then there exists a unique weak solution $u\in W^{1,p}(\Omega,\mathcal{D};\widetilde{\omega}^p dx)$ of \eqref{eq:eqw}. Moreover, we have the estimate
	\begin{align*}
	\int_{\Omega}|\nabla v|^p\widetilde{\omega}^p\,dx \leq c\int_{\Omega}|\widetilde{G}|^p\widetilde{\omega}^p\,dx
	\end{align*}
	with $c=c(n,p,\Lambda)$.
\end{theorem}
Now we give the main assumption for higher integrability result in this section. Recall $\gamma_1$ as in \eqref{eq:gamma1}. Let us define
\begin{align*}
q:=
\begin{cases}
\max\left\{\gamma_1,\frac{p+1}{2}\right\}&\quad\text{if }p\in\left(1,\frac{n}{n-1}\right],\\
\max\left\{\gamma_1,\frac{np}{n+p}\right\}&\quad\text{if }p\in\left(\frac{n}{n-1},\infty\right).
\end{cases}
\end{align*}
Then we consider the following assumption.
\begin{assumption}
Let $r_0>0$ be given. for any $x_0\in \mathcal{D}$ and $r\in(0,r_0)$, we have
\begin{align}\label{eq:cen'}
\dfrac{\textrm{Cap}_{q,\omega^q}(\mathcal{D}\cap Q_{r}(x_0),\setR^n)}{\omega^q(Q_r(x_0))}\geq c_0 r^{-q}
\end{align}
for some $c_0>0$ independent of $x_0$ and $r$.	
\end{assumption}

Similar to Theorem \ref{thm:hi}, we have the following result for the problem \eqref{eq:eqw}.

\begin{theorem}[Higher integrability]
	Let $\Omega\subset\setR^n$ be a bounded $(L,R_0)$-Lipschitz domain, $1<p<\infty$, $\widetilde{\omega}^p\in\mathcal{A}_p(\Omega)$, and assume \eqref{eq:cen'}. If $G\in L^{p+{\delta_0}}(\Omega;\widetilde{\omega}^{p+\delta_0} dx)$ with $\delta_0>0$, then there exist $\delta_1=\delta_1(n,\delta_0)>0$ such that the following holds: the weak solution $u\in W^{1,p}(\Omega,\mathcal{D};\widetilde{\omega}^p dx)$ of \eqref{eq:eqw} satisfies $\abs{\nabla u}\in L^{p+\delta}(\Omega;\widetilde{\omega}^{p+\delta} dx)$ for any  $\delta\leq\delta_1$. Moreover, the estimate
	\begin{align}\label{eq:hi'}
	\int_{\Omega}|\nabla u|^{p+\delta}\widetilde{\omega}(x)^{p+\delta}\,dx\leq c\int_{\Omega}|G|^{p+\delta}\widetilde{\omega}(x)^{p+\delta}\,dx
	\end{align}
	holds for some $c=c(n,p,\Lambda,[\widetilde{\omega}^p]_{\mathcal{A}_p},q,L,R_0,|\Omega|,c_0,r_0)>0$.
\end{theorem}
\begin{proof}
	Using the transformation $y=\Psi(x)$, similarly to \eqref{eq:eq2}, the problem \eqref{eq:eqw} locally takes of the form
	\begin{align}\label{eq:eq2'}
	\begin{split}
	&-\divergence\left((\nabla\Psi(y))^t\widetilde{\setM}^2(y)\abs{\widetilde{\setM}(y)\nabla\Psi(y)\nabla \widetilde{u}(y)}^{p-2}(\nabla\Psi(y)\nabla \widetilde{u}(y))\right)\\
	&\quad\quad\quad\quad\quad=-\divergence\left((\nabla\Psi(y))^t\abs{\widetilde{\setM}(y)\widetilde{G}(y)}^{p-2}\widetilde{\setM}(y)^2\widetilde{G}(y)\right)\quad\text{in }\widetilde{Q}^+_R
	\end{split}
	\end{align}
	with
	\begin{align*}
	\begin{split}
	\widetilde{u}=0\quad\text{on }\widetilde{\mathcal{D}}_R\quad\text{and}\quad
	\dfrac{\partial \widetilde{u}}{\partial \widetilde{\nu}}=0\quad\text{on }\widetilde{\mathcal{N}}_{R}.
	\end{split}
	\end{align*}
	Here, we note the following:
	\begin{itemize}
		\item We have $\widetilde{\setM}(y):=(\setM\circ\Psi^{-1})(y)$ and $\widetilde{\omega}(x):=\abs{\widetilde{\setM}(y)}\in \mathcal{A}_p$.
		\item Recall $\widetilde{\mathcal{D}}_{R}=\Psi(\mathcal{D}\cap Q_R)\cap \widetilde{Q}_{R}$ and $\widetilde{\mathcal{N}}_{R}=\Psi(\mathcal{N}\cap Q_{R})\cap \widetilde{Q}_{R}$ as in the proof of Theorem \ref{thm:hi}.
		\item Note that $\widetilde{u}(y):=(u\circ\Psi^{-1})(y)\in W^{1,p}(\widetilde{Q}^+_{R},\widetilde{\mathcal{D}}_R;\widetilde{\omega}^p dy)$, and $\widetilde{G}(y):=(G\circ\Psi^{-1})(y)\in L^p(\widetilde{Q}^+_{R};\widetilde{\omega}^p dy)$ from $\psi$ being a Lipschitz map and \eqref{eq:ellip}.
		\item The outer conormal derivative of $\widetilde{u}$ is denoted by $\frac{\partial\widetilde{u}}{\partial{\tilde{\nu}}}$.
	\end{itemize}

	Now we apply the even extension for the weight $\widetilde{\omega}(y)$ and $\widetilde{u}$ with respect to $\{y:y_n=0\}$. We write the extended function of $\widetilde{\omega}$ as $\widehat{\omega}$, and $\widetilde{u}$ as $\widehat{u}$. Then $\widehat{\omega}^p\in\mathcal{A}_p$ and $\widehat{u}\in W^{1,p}(\widetilde{Q}_{R},\widetilde{\mathcal{D}}_R;\widehat{\omega}^p dy)$ satisfies the following equation:
	\begin{align}\label{eq:eq3'}
	\begin{split}
	-\divergence\left(\abs{\widehat{\nabla} \widehat{u}(y)}^{p-2}\widehat{\setM}(y)^2\nabla \widehat{u}(y)\right)=-\divergence\left(\widehat{G}(y)\right)\quad\text{in }\widetilde{Q}_R
	\end{split}
	\end{align}
	with
	\begin{align*}
	\begin{split}
	\widehat{u}=0\quad\text{on }\widetilde{\mathcal{D}}_R.
	\end{split}
	\end{align*}
	Here, we list definitions of each quantities and properties:
	\begin{itemize}
		\item We define $\widehat{\nabla} \widehat{u}(y)$ as an odd extension of $\widetilde{\setM}(y)\nabla\Psi(y)\nabla \widetilde{u}(y)$. Since $\psi$ is Lipschitz, $|\widehat{\nabla} \widehat{u}(y)|\eqsim |\nabla \widetilde{u}(y)|\widehat{\omega}(y)$.
		\item Note that $A(y):=(\nabla\Psi(y))^t\widetilde{\setM}^2(y)(\nabla\Psi(y))$ is a symmetric matrix so that $A(y)^{\frac{1}{2}}$ exists. Then the weighted coefficient matrix $(\widehat{\setM}_{ij}(y))_{1\leq i,j\leq n}=\widehat{\setM}(y)$ is such that $\widehat{\setM}_{in}(y)=\widehat{\setM}_{ni}(y)$ for $i\neq n$ are odd extensions of $A(y)^{\frac{1}{2}}_{in}$, and all other elements of $\widehat{\setM}_{ij}(y)$ are even extensions of $A(y)^{\frac{1}{2}}_{ij}$. Note that $\widehat{\setM}(y)$ satisfies $\widetilde{\omega}(x)|\xi|\eqsim |\setM(x)\xi|$ for any $\xi\in\setR^n$ which is similar to \eqref{eq:ellip}.
		\item We write $\widehat{G}(y)=(\widehat{G}_1(y),\dots,\widehat{G}_n(y))\in L^p(\widetilde{Q}^+_{R};\widetilde{\omega}^p dy)$, where $\widehat{G}_i(y)$ are even extensions of $[(\nabla\Psi(y))^t\abs{\widetilde{\setM}(y)\widetilde{G}(y)}^{p-2}\widetilde{\setM}(y)^2\widetilde{G}(y)]_i(y)$ for $i\neq n$, and $\widehat{G}_n(y)$ are the odd extension of $[(\nabla\Psi(y))^t\abs{\widetilde{\setM}(y)\widetilde{G}(y)}^{p-2}\widetilde{\setM}(y)^2\widetilde{G}(y)]_n(y)$. Note that
		\begin{align*}
		|\widehat{G}|\eqsim \widehat{\omega}^{p}|\widetilde{G}|^{p-1}.
		\end{align*}
	\end{itemize}
	Then $\widehat{u}\in W^{1,p}(\widetilde{Q}_{R},\widetilde{\mathcal{D}}_R;\widehat{\omega}^p dy)$ from \eqref{eq:eq3'} satisfies
	\begin{align}\label{eq:eq5'}
	\int_{\widetilde{Q}_{R}}\abs{\widehat{\nabla} \widehat{u}}^{p-2}(\widehat{\setM}\nabla \widehat{u})\cdot(\widehat{\setM}\nabla \phi)\,dy=\int_{\widetilde{Q}_{R}}\widehat{G}\cdot \nabla \phi\,dy
	\end{align}
	for every $\phi\in W^{1,p}_0(\widetilde{Q}_{R},\widetilde{\mathcal{D}}_R;\widehat{\omega}^p dy)$.

	From now on, we assume
	\begin{align*}
	y_0\in \widetilde{Q}_{R/2}\setminus\partial \widetilde{Q}_{R/2},\quad\text{and }r\leq\frac{1}{2}\distance(y_0,\partial \widetilde{Q}_{R/2}).
	\end{align*}
First, consider the case of $Q_{3r/2}(y_0)\subset \widetilde{Q}_{R}\setminus \widetilde{\mathcal{D}}_{R}$. Test $\phi=(\widehat{u}-\lambda)\eta^p$ in \eqref{eq:eq5'} for $\lambda=\mean{\widehat{u}}_{Q_{3r/2}(y_0)}$, and the cutoff function $\eta\in C^{\infty}_0(Q_{3R/2}(y_0))$ satisfies $0\leq \eta\leq 1$, $\eta\equiv 1$ in $Q_{r}(y_0)$ and $|\nabla \eta|\leq c/r$. Then \eqref{eq:eq5'}, Young's inequality, and $|\widehat{G}|\eqsim \widehat{\omega}^{p}|\widetilde{G}|^{p-1}$ imply the following Caccioppoli inequality
	\begin{align}\label{eq:hi3'}
	\int_{Q_{r}(y_0)}|\nabla \widehat{u}|^p\widehat{\omega}^p \,dy\leq c(n,p,\Lambda)\left(\int_{Q_{3r/2}(y_0)}\left|\frac{\widehat{u}-\lambda}{r}\right|^p\widehat{\omega}^p\,dy+\int_{Q_{3r/2}(y_0)}|\widehat{G}|^p\widehat{\omega}^p\,dy\right).
	\end{align}
	Moreover, from Lemma \ref{lem:SP0'},
	\begin{align*}
	\left(\dashint_{Q_{3r/2}(y_0)}\left|\frac{\widehat{u}-\lambda}{r}\right|^{p}\widehat{\omega}^p\,dx\right)^{1/{p}}\leq c\left(\dashint_{Q_{3r/2}(y_0)}|\nabla \widehat{u}|^{q}\widehat{\omega}^{q}\,dx\right)^{1/{q}}
	\end{align*}
	and so we have
	\begin{align}\label{eq:hi3.1'}
	\begin{split}
	&\left(\dashint_{Q_r(y_0)}|\nabla \widehat{u}|^p\widehat{\omega}^p\,dy\right)^{1/p}\\
	&\quad\leq c\left(\left(\dashint_{Q_{2r}(y_0)}|\nabla \widehat{u}|^{q}\widehat{\omega}^{q}\,dy\right)^{1/q}+\left(\dashint_{Q_{2r}(y_0)}|\widehat{G}|^p\widehat{\omega}^p\,dy\right)^{1/p}\right).
	\end{split}
	\end{align}

	In case of $Q_{3r/2}(y_0)\cap \widetilde{\mathcal{D}}_{R}\neq \emptyset$, we choose a test function $\phi=u\eta^p$ in \eqref{eq:eq5'} and so we get
	\begin{align}\label{eq:hi4'}
	\int_{Q_r(y_0)}|\nabla \widehat{u}|^p\widehat{\omega}^p\,dy\leq c(n,p,L)\left(\int_{Q_{2r}(y_0)}\left|\frac{\widehat{u}}{r}\right|^p\widehat{\omega}^p\,dy+\int_{Q_{2r}(y_0)}|\widehat{G}|^p\widehat{\omega}^p\,dy\right).
	\end{align}
	Now we apply Proposition \ref{prop:SP'} with the assumption \eqref{eq:cen'}. With the same argument for obtaining \eqref{eq:cen2}, Proposition \ref{prop:SP'} is applicable similarly and so it follows that
	\begin{align*}
	\left(\dashint_{Q_{2r}(y_0)}\left|\frac{\widehat{u}}{r}\right|^p\widehat{\omega}^p\,dy\right)^{\frac{1}{p}}\leq c(L,n,c_0)\left(\dashint_{Q_{2r}(y_0)}|\nabla \widehat{u}|^{q}\widehat{\omega}^{q}\,dy\right)^{\frac{1}{q}}.
	\end{align*}
	Combining with \eqref{eq:hi4'}, we again obtain \eqref{eq:hi3.1'}. Now we use Gehring's lemma \cite{KriStr18} to find
	\begin{align*}
	\begin{split}
	&\left(\dashint_{Q_r(y_0)}|\nabla \widehat{u}|^{p+\delta}\widehat{\omega}^{p+\delta}\,dy\right)^{1/(p+\delta)}\\
	&\quad\leq c\left(\left(\dashint_{Q_{2r}(y_0)}|\nabla \widehat{u}|^{p}\widehat{\omega}^{p}\,dy\right)^{1/p}+\left(\dashint_{Q_{2r}(y_0)}|\widehat{G}|^{p+\delta}\widehat{\omega}^{p+\delta}\,dy\right)^{1/(p+\delta)}\right).
	\end{split}
	\end{align*}
	In other words, we have
	\begin{align*}
	\begin{split}
	&\left(\int_{Q_r(y_0)}|\nabla \widehat{u}|^{p+\delta}\widehat{\omega}^{p+\delta}\,dy\right)^{1/(p+\delta)}\\
	&\quad\leq c\frac{1}{r^{n(1/ps-1/p)}}\left(\int_{Q_{2r}(y_0)}|\nabla \widehat{u}|^{p}\widehat{\omega}^{p}\,dy\right)^{1/p}+c\left(\int_{Q_{2r}(y_0)}|\widehat{G}|^{p+\delta}\widehat{\omega}^{p+\delta}\,dy\right)^{1/(p+\delta)}.
	\end{split}
	\end{align*}
	Then the above estimate implies
	\begin{align*}
	\begin{split}
	&\left(\int_{Q_r(y_0)}|\nabla \widetilde{u}|^{p+\delta}\widetilde{\omega}^{p+\delta}\,dy\right)^{1/(p+\delta)}\\
	&\quad\leq c\left(\int_{Q_{2r}(y_0)}|\nabla \widetilde{u}|^{p}\widetilde{\omega}^{p}\,dy\right)^{1/p}+c\left(\int_{Q_{2r}(y_0)}|\widetilde{G}|^{p+\delta}\widetilde{\omega}^{p+\delta}\,dy\right)^{1/(p+\delta)}
	\end{split}
	\end{align*}
	with $c=c(n,p,[\omega^{p}]_{\mathcal{A}_p},c_0,L,R_0)>0$, since the function $\widehat{u}$ was the even extension of $\widetilde{u}$ with respect to $\{y:y_n=0\}$, together with the definition of $\widehat{\omega}$ and $\widehat{G}$ from $\widetilde{\omega}$ and $\widetilde{G}$, respectively. Now by the covering argument together with the interior estimate
	\begin{align*}
	\begin{split}
	&\left(\int_{Q_r}|\nabla \widetilde{u}|^{p+\delta}\widetilde{\omega}^{p+\delta}\,dy\right)^{1/(p+\delta)}\\
	&\quad\leq c\left(\int_{Q_{2r}}|\nabla \widetilde{u}|^{p}\widetilde{\omega}^{p}\,dy\right)^{1/p}+c\left(\int_{Q_{2r}}|\widetilde{G}|^{p+\delta}\widetilde{\omega}^{p+\delta}\,dy\right)^{1/(p+\delta)}
	\end{split}
	\end{align*}
	with $Q_{2r}\Subset\Omega$, the proof is completed.
\end{proof}

\printbibliography

\end{document}